\documentclass[11pt]{article}
\usepackage[numbers,sort&compress]{natbib}
\usepackage{enumerate}
\usepackage{amscd}
\usepackage{amsmath}
\usepackage{latexsym}
\usepackage{amsfonts}
\usepackage{setspace}
\usepackage{amssymb}
\usepackage{amsthm}
\usepackage{verbatim}
\usepackage{mathrsfs}
\usepackage{enumerate}
\usepackage[hypertexnames=false]{hyperref}
\usepackage[utf8]{inputenc}
\usepackage{amsmath, amssymb, pgfplots}
\pgfplotsset{compat = newest, width = 12cm, height =12cm}
\usepackage{tikz}
\usepackage{caption}

\theoremstyle{plain}
\theoremstyle{definition}\newtheorem{theorem}{Theorem}[section]
\theoremstyle{definition}\newtheorem{lemma}[theorem]{Lemma}
\theoremstyle{plain}\newtheorem{coro}[theorem]{Corollary}
\theoremstyle{plain}\newtheorem{prop}[theorem]{Proposition}
\theoremstyle{definition}\newtheorem{remark}{Remark}[section]
\usepackage{xcolor}
\usepackage{mdframed}

\newcommand{\diver}{\text{div}\,}

\newcommand{\define}{\stackrel{\mathrm{def}}{=}}

\allowdisplaybreaks

\numberwithin{equation}{section}
\begin{document}
	\title{Norm Inflation for Inviscid and Fully Dissipative Boussinesq Systems in Supercritical Spaces}
	\author{Qionglei Chen\footnote{Institute of Applied Physics and Computational Mathematics, Beijing, 100088, China. Email: chen\_qionglei@iapcm.ac.cn}~~~\,\,\,\,Yaowei Xie\footnote{Institute of Applied Physics and Computational Mathematics, Beijing, 100088, China. Email: mathxyw@163.com}}
	\date{}
	\maketitle
	\begin{abstract}
We prove norm inflation, in the sense of strong ill-posedness, for the two-dimensional
Boussinesq system in supercritical Besov spaces. For the inviscid system, norm
inflation holds in
\(\dot B^\beta_{p,q}(\mathbb R^2)\times
\dot B^\beta_{p,r}(\mathbb R^2)\)
for \(\beta\neq0\), \(1<p\leq\infty\), \(1\leq q,r\leq\infty\), and
\(-2<\beta-\frac{2}{p}<1\). For the fully dissipative system, the same conclusion
holds in the range \(-2<\beta-\frac{2}{p}<-1\). In both cases, the results cover
almost all supercritical Besov spaces satisfying the local integrability
condition.
Norm inflation occurs in the density component \(\rho\), while the velocity
component \(u\) remains bounded. In the fully dissipative case, the inflation space is
supercritical for \(u\), but subcritical for \(\rho\) with respect to its own scaling.
This is not a contradiction: the density is transported by a velocity field in a supercritical regime, and this transport mechanism is precisely what produces norm inflation in \(\rho\).

	\end{abstract}
	
	\section{Introduction}
This paper is concerned with the following two-dimensional (2D) Boussinesq system
\begin{align}\label{boussinesq-orginal-0}
		\begin{cases}
		\partial_t u-\mu\Delta u+u \cdot \nabla u+\nabla p=\rho e_2,\\[1mm]
		\partial_t \rho-\nu\Delta \rho+u\cdot \nabla \rho=0,  \\[1mm]
		\text{div}\, u=0, \\[1mm]
		u(x, 0)=u_0(x),\,\, \rho(x, 0)=\rho_0(x),
	\end{cases}
\end{align}
where $u(x,t)\colon \mathbb{R}^2\times \mathbb{R}_+ \to \mathbb{R}^2$, 
$\rho(x,t)\colon \mathbb{R}^2\times \mathbb{R}_+ \to \mathbb{R}$, 
and $p(x,t)\colon \mathbb{R}^2\times \mathbb{R}_+ \to \mathbb{R}$ 
denote the velocity field, the density (or temperature), and the pressure, respectively. 
Here $\mu\geq 0$ and $\nu\geq 0$ are the viscosity and diffusion coefficients, respectively, and $e_2=(0,1)$ denotes the unit vector in the vertical direction. The Boussinesq equations form a classical model in fluid dynamics, arising in atmospheric fronts, oceanic circulation, Rayleigh--Bénard convection, and stratified flows, where the scalar quantity may represent either temperature or density variations; see, e.g., \cite{m-phys-boussinesq-Atmosphere and Ocean,v-phys-boussinesq-2017-large-scale}. 

In this paper, we focus on two choices of the parameters in \eqref{boussinesq-orginal-0}: 
$\mu=\nu=0$, corresponding to the inviscid Boussinesq system, and $\mu,\nu>0$, corresponding to the fully dissipative Boussinesq system. 
For the fully dissipative case, we restrict attention to the choice $\mu=\nu=1$. Both systems have been extensively studied over the past few decades. We next review the relevant literature and state our main results for each case.

\textbf{The inviscid case.}
If \(\rho\equiv0\), the inviscid Boussinesq system reduces to the classical
two-dimensional incompressible Euler equations. For the Euler equations, the local
well-posedness theory is closely related to the Lipschitz regularity of the velocity
field; see, for example,
\cite{k-KP-commutator-Euler,c-chae-Euler-local-besov}. The corresponding local
theory for the inviscid Boussinesq system was established in
\cite{c-CN-local-inviscid}. In particular, the inviscid Boussinesq system is locally
well posed in $B^{1+\frac dp}_{p,1}(\mathbb R^d)
\times
B^{1+\frac dp}_{p,1}(\mathbb R^d)$ with $1<p<\infty.$

The question of finite-time singularity formation for smooth finite-energy solutions
of the inviscid Boussinesq system on \(\mathbb R^2\) or \(\mathbb T^2\) remains widely
open. Existing finite-time blowup results rely mainly on boundaries, special geometric
structures, or suitable forcing mechanisms; see, for instance, the works of Elgindi
and Jeong \cite{e-EJ-boussinesq-finitetime-jiao}, Chen and Hou
\cite{c-CH-2021-boussinesq-half-plane,c-CH-2022-boussinesq-arxiv-self-similar},
and C{\'o}rdoba, La{\'{\i}}n-Sanclemente, and Mart{\'{\i}}nez-Zoroa
\cite{c-CLM-boussinesq-finitetime-singular-2d-Boussinesq}.

We next turn to ill-posedness phenomena. In the critical regularity regime,
Bourgain and Li \cite{b-BL-euler-2-illpose-bourgain-2015-IA} established strong
ill-posedness for the two-dimensional Euler equations in the critical Sobolev space
\(H^2(\mathbb R^2)\). Subsequent results on strong ill-posedness in critical spaces
can be found in
\cite{e-EM-elgindi2020infty,
	b-BL-euler-4-illpose-bourgain-2021-IMRN,
	m-MY-euler-illposed,
	e-EJ-euler-simper,
	k-KJ-Euler-simple illposed}.

For the inviscid Boussinesq system, the ill-posedness theory is much less developed.
Elgindi and Masmoudi \cite{e-EM-elgindi2020infty} proved mild ill-posedness near
the stably stratified steady state \(\rho_{\mathrm{eq}}=-y\) in the critical
\(L^\infty\) regime, namely for
\(\omega=\nabla\times u, \nabla\rho \in L^\infty\),
which corresponds to the Lipschitz regularity of \(u\) and \(\rho\). Their instability
mechanism is reflected in the growth of the vorticity \(\omega\). More recently,
Bianchini, Hientzsch, and Iandoli
\cite{b-BHL-strong-2d-infty-boussinesq-3dEuler}
proved strong ill-posedness in the same critical \(L^\infty\) framework, establishing
norm inflation for \(\partial_x\rho\) while keeping \(\omega\) and \(\partial_y\rho\)
bounded.

Recent years have also seen important progress on strong ill-posedness in
supercritical spaces. Luo
\cite{l-L-euler-10-illpose-lxyt-2024-arxiv}
proved strong ill-posedness for the three-dimensional Euler equations in
supercritical Sobolev spaces \(H^s(\mathbb R^3)\), \(0<s<\frac52\). In a different
direction, C{\'o}rdoba, Mart{\'{\i}}nez-Zoroa, and Oza{\'n}ski
\cite{c-CMO-euler-2d-Instantaneous gap loss of Sobolev regularity-24}
constructed solutions of the two-dimensional Euler equations with instantaneous
loss of Sobolev regularity in the supercritical range \(H^s(\mathbb R^2)\),
\(1<s<2\). Their construction also yields norm inflation as a byproduct.

These developments naturally lead to the study of strong ill-posedness for the
inviscid Boussinesq system in supercritical spaces. Compared with the Euler theory,
the inflation mechanism here is carried by the density component through the coupling
between velocity and density, rather than by the velocity field itself. The construction
also covers both positive and negative Besov regularities in the locally integrable
supercritical range.

More precisely, our first result establishes strong ill-posedness for the inviscid Boussinesq system in
\[
\dot B^\beta_{p,q}(\mathbb R^2)
\times
\dot B^\beta_{p,r}(\mathbb R^2),
\qquad
\beta\neq0,\quad
1<p\leq\infty,\quad 1\leq q,r\leq\infty,\quad
-2<\beta-\frac{2}{p}<1.
\]
This is made precise in the following theorem.

\begin{theorem}[Norm inflation for the inviscid Boussinesq system]\label{thm}
		Fix $1<p\leq\infty$ and $\beta\neq0$ such that
		$-2<\beta-\frac{2}{p}<1$. For any $\varepsilon>0$, there exists a solution
		$(u,\rho)$ to \eqref{boussinesq-orginal-0} in the inviscid case $\mu=\nu=0$, with initial data
		$(u_0,\rho_0)\in C_c^\infty(\mathbb R^2)$ such that
		\begin{align*}
			\|u_0\|_{\dot{B}^{\beta}_{p,1}}+\|\rho_0\|_{\dot{B}^{\beta}_{p,1}}< \varepsilon.
	\end{align*}
		Moreover, there exists a time $0 < t^* < \varepsilon$ such that
	\begin{align*}
		\|\rho(t^*)\|_{\dot{B}^{\beta}_{p,\infty}}>\frac{1}{\varepsilon}.
	\end{align*}
\end{theorem}

Since $\dot B^\beta_{p,1}\hookrightarrow\dot B^\beta_{p,q}$ and
$\|f\|_{\dot B^\beta_{p,r}}\geq\|f\|_{\dot B^\beta_{p,\infty}}$, Theorem
\ref{thm} gives the stated norm inflation for every $1\leq q,r\leq\infty$.

\textbf{The fully dissipative case.}
If \(\rho\equiv0\), the fully dissipative Boussinesq system reduces to the
incompressible Navier--Stokes equations. In this setting, the critical regularity is
determined by the parabolic scaling
\[
u_\lambda(t,x)=\lambda u(\lambda^2t,\lambda x),
\qquad
p_\lambda(t,x)=\lambda^2p(\lambda^2t,\lambda x),
\]
which leads to the critical Besov spaces $\dot B^{-1+\frac dp}_{p,r}(\mathbb R^d), 1\le p,r\le\infty.$ The local well-posedness theory for the Navier--Stokes equations in critical spaces
is classical; we refer to
\cite{b-BCD-b-bahouri-chemin-fourier-book-2011}
and the references therein.

For the fully dissipative Boussinesq system, the natural scaling is
\[
u_\lambda(t,x)=\lambda u(\lambda^2t,\lambda x),
\qquad
\rho_\lambda(t,x)=\lambda^3\rho(\lambda^2t,\lambda x),
\qquad
p_\lambda(t,x)=\lambda^2p(\lambda^2t,\lambda x).
\]
Accordingly, the scale-invariant regularities are formally given by
\begin{align}\label{critical-fully-diss}
	u_0\in \dot B^{-1+\frac{2}{p}}_{p,q}(\mathbb R^2),
	\qquad
	\rho_0\in \dot B^{-3+\frac{2}{p}}_{p,r}(\mathbb R^2),
\end{align}
where \(1\le p,q,r\le\infty\). The local theory for the fully dissipative
Boussinesq system closely parallels that of Navier--Stokes; see
\cite{m-MB01-Vorticity and Incompressible Flow}. In particular, one has local
	well-posedness in critical Besov product spaces, for instance $\dot B^{-1+\frac dp}_{p,q}(\mathbb R^d)
\times
\dot B^{-1+\frac dp}_{p,q}(\mathbb R^d)$ with $1\leq p<\infty,$
where both components are measured at the velocity critical regularity.

We next recall some ill-posedness results in critical spaces. For the
Navier--Stokes equations, Bourgain and Pavlovi\'c
\cite{b-BP08-NS-B-1infty-illposed}
established strong ill-posedness in
\(\dot B^{-1}_{\infty,\infty}\), and this mechanism was further developed in
\cite{w-Wang15-ill-NS-B-1,
	y-Yoneda10-ill-NS-BMO-1}.
For the fully dissipative Boussinesq system, Li and Wang
\cite{l-LW-norm-boussinesq-B-1-inftyinfty}
proved strong ill-posedness on \(\mathbb T^3\) with small initial data in $\dot B^{-1}_{\infty,\infty}
\times
\dot B^{-1}_{\infty,\infty},$
following the Bourgain--Pavlovi\'c norm inflation strategy.

Recent years have also seen important progress on strong ill-posedness in
supercritical spaces. Luo
\cite{l-L-euler-10-illpose-lxyt-2024-arxiv}
proved strong ill-posedness for the three-dimensional Navier--Stokes equations in
supercritical Sobolev spaces \(H^s(\mathbb R^3)\), \(0<s<\frac12\). This was later
extended in
\cite{l-L-euler-101-illpose-lxyt-2025-arxiv}
to supercritical Besov spaces
\(\dot B^s_{p,q}(\mathbb R^3)\) with
\(s\neq0\), \(1\le p,q\le\infty\), and
\(-3<s-\frac{3}{p}<-1\), including negative regularity regimes near the scaling
critical threshold.

These results suggest the corresponding question for the fully dissipative
Boussinesq system. In contrast to the Navier--Stokes case, the inflation mechanism
here is generated through the density component in a coupled velocity and density
system. Moreover, the admissible range is dictated by the parabolic scaling of the
velocity equation, which shifts the upper threshold by two derivatives compared with
the inviscid case.

More precisely, we prove strong ill-posedness for the fully dissipative Boussinesq system in
\[
\dot B^\beta_{p,q}(\mathbb R^2)
\times
\dot B^\beta_{p,r}(\mathbb R^2),
\qquad
\beta\neq0,\quad
1<p\leq\infty,\quad 1\leq q,r\leq\infty,\quad
-2<\beta-\frac{2}{p}<-1.
\]
In view of the scaling \eqref{critical-fully-diss}, this range is supercritical for the
velocity component \(u\), but subcritical for the density component \(\rho\) with
respect to its own scaling. Nevertheless, norm inflation occurs in \(\rho\), because
the density is transported by a velocity field in a supercritical regime. This is made
precise in the following theorem.

\begin{theorem}[Norm inflation for the fully dissipative Boussinesq system]\label{thm-1}
		Fix $1<p\leq\infty$ and $\beta\neq0$ such that
		$-2<\beta-\frac{2}{p}<-1$. For any $\varepsilon>0$, there exists a solution
		$(u,\rho)$ to \eqref{boussinesq-orginal-0} in the fully dissipative case $\mu=\nu=1$, with initial data
		$(u_0,\rho_0)\in C_c^\infty(\mathbb R^2)$ such that
		\begin{align*}
			\|u_0\|_{\dot{B}^{\beta}_{p,1}}+\|\rho_0\|_{\dot{B}^{\beta}_{p,1}}< \varepsilon.
	\end{align*}
		Moreover, there exists a time $0 < t^* < \varepsilon$ such that
	\begin{align*}
		\|\rho(t^*)\|_{\dot{B}^{\beta}_{p,\infty}}>\frac{1}{\varepsilon}.
	\end{align*}
\end{theorem}

The same Besov embeddings show that Theorem \ref{thm-1} gives the stated conclusion
for every $1\leq q,r\leq\infty$.

\begin{remark}
	The upper thresholds in our results are consistent with the corresponding local well-posedness theory. 
		For the inviscid Boussinesq system, local well-posedness holds in
	$B^{1+\frac{2}{p}}_{p,1}(\mathbb R^2)\times B^{1+\frac{2}{p}}_{p,1}(\mathbb R^2)$ with $1<p<\infty$; see \cite{c-CN-local-inviscid}. 
	Thus the condition $\beta-\frac{2}{p}<1$ is sharp with respect to this threshold. 
		For the fully dissipative system, local well-posedness holds in $
	\dot B^{-1+\frac dp}_{p,q}(\mathbb R^d)\times
	\dot B^{-1+\frac dp}_{p,q}(\mathbb R^d)$ with $1\leq p<\infty,1\leq q<\infty$;
	see \cite{m-MB01-Vorticity and Incompressible Flow,b-BCD-b-bahouri-chemin-fourier-book-2011}. 
Thus the upper threshold
\(\beta-\frac{2}{p}<-1\) is sharp in the velocity critical scale.
\end{remark}

\begin{remark}
The restriction \(\beta-\frac{2}{p}>-2\) ensures that the Besov spaces considered here lie in the locally integrable range. Consequently, the initial data produced by the construction are genuine locally integrable functions, rather than merely distributions.
\end{remark}

\begin{remark}
	For the fractionally dissipative Boussinesq system
	\[
		\begin{cases}
			\partial_t u+\Lambda^\alpha u+u\cdot\nabla u+\nabla p=\rho e_2,\\
			\partial_t \rho+\Lambda^\gamma\rho+u\cdot\nabla\rho=0,\\
		\diver u=0,
	\end{cases}
	\]
where $\alpha,\gamma\geq 0$, $\Lambda=(-\Delta)^{\frac{1}{2}}$, the case \(\alpha+\gamma=1\) is usually called the critical dissipation regime. Existing
		global regularity results in the critical dissipation regime are
	typically proved in high regularity spaces, such as \(H^s(\mathbb R^2)\), \(s>2\);
	see
	\cite{j-JMWZ-boussinesq-critical dissipation-14siam,
		s-SW-critical dissipation-2019-JAM,
		s-SWXY-fractional-bouss-2025-math ann}.
	The argument of this paper also applies to this fractionally dissipative system and
gives strong ill-posedness in \(H^s(\mathbb R^2)\) for
\[
-1<s<2-\max\{\alpha,\gamma\},\qquad s\neq0.
\]
Thus there remains a natural gap between this range and the known high
regularity global theory above \(2\).
\end{remark}

\textbf{Idea of the proof.}
We briefly explain the main idea of the construction. Our construction is motivated by shear mechanisms developed for Euler and SQG
equations
\cite{c-CMO-euler-2d-Instantaneous gap loss of Sobolev regularity-24,
	c-CMZ-SQG-Ck-Sobolev-nonexistence-22,
	c-CMZ-SQG-cmp-generalized-24,g-GL-IMRN-SQG-critical,c-CMZ-SQG-annpde-fractional diffusion},
and by the backward energy cascade used to handle negative regularity in the
Navier--Stokes construction of
\cite{l-L-euler-101-illpose-lxyt-2025-arxiv}. The point here is to implement these
ideas in the Boussinesq system, where the velocity and the density are coupled.

The approximate solution consists of a stationary radial vorticity and a density
transported by the velocity induced by this vorticity. We work in polar coordinates
\begin{align*}
	x=(r\cos\theta,r\sin\theta),
	\qquad r\geq 0,\quad \theta\in[0,2\pi).
\end{align*}
Let
\begin{align*}
	\bar w(x,t)
	=
	\lambda^{\frac{2}{p}+1-\beta}
	(\log\log\lambda)^{-\frac{|\beta|}{2+|\beta|}}
	f(\lambda r),
	\qquad
	\bar u=u[\bar w],
\end{align*}
where
\begin{align*}
	u[f](x)
	=
	\frac{1}{2\pi}
	\int_{\mathbb R^2}
	\frac{(x-y)^\perp}{|x-y|^2}f(y)\,dy
\end{align*}
is the Biot--Savart law. Since \(\bar w\) is radial and independent of time,
\(\bar u\) is stationary and purely angular $
	\bar u(r,\theta)=\bar u_\theta(r)e_\theta.$ We denote its angular velocity by
\begin{align*}
	\Omega_\lambda(r)
	:=
	\frac{\bar u_\theta(r)}{r}
	=
	\lambda^{\frac{2}{p}-\beta}
	(\log\log\lambda)^{-\frac{|\beta|}{2+|\beta|}}
	\frac{u_\theta[f](\lambda r)}{r}.
\end{align*}
Thus, the corresponding flow preserves \(r\) and maps
\begin{align*}
	(r,\theta)
	\longmapsto
	\bigl(r,\theta+t\Omega_\lambda(r)\bigr).
\end{align*}

Set
\begin{align*}
	t^*
	=
	\lambda^{-\frac{2}{p}-1+\beta}\log\log\lambda,
\end{align*}
which is the time at which norm inflation will occur. For \(\beta<0\), we
additionally set $
	k=-\lfloor\beta\rfloor+2,$
where \(\lfloor\beta\rfloor\) denotes the greatest integer not exceeding
\(\beta\). We define the initial density profile by
\begin{align*}
	\bar\rho_0(r,\theta)
	=
	\begin{cases}
		\displaystyle
		\lambda^{\frac{2}{p}-\beta}
		(\log\log\lambda)^{-\frac{\beta}{2+\beta}}
		g(\lambda r)\cos\theta,
		& \beta>0,\\[3mm]
		\displaystyle
		\lambda^{\frac{2}{p}-\beta-k}
		(\log\log\lambda)^{-\frac{\beta+2k}{2-\beta}}
		\left.
		\partial_r^k
		\left[
		g(\lambda r)
		\cos\left(
		\vartheta+t^*\Omega_\lambda(r)
		\right)
		\right]
		\right|_{\vartheta=\theta},
		& \beta<0,
	\end{cases}
\end{align*}
where, in the second case, the radial derivatives are taken with
\(\vartheta\) fixed. The approximate density is then defined uniformly by
\begin{align*}
	\bar\rho(t,r,\theta)
	=
	\bar\rho_0\bigl(r,\theta-t\Omega_\lambda(r)\bigr).
\end{align*}
 By construction, the approximate pair satisfies
\begin{align*}
	\partial_t\bar w+\bar u\cdot\nabla\bar w=0,~~
	\partial_t\bar\rho+\bar u\cdot\nabla\bar\rho=0.
\end{align*}

The radial profiles \(f\) and \(g\) are chosen as in Lemma
\ref{lem-choice-fg-inviscid}. Their supports are separated, \(f\) generates
a nondegenerate angular shear on the support of \(g\), and
\begin{align*}
	f^{(-2)}\in C_c^\infty((0,\infty)).
\end{align*}
The shear produces rapid oscillations in \(\bar\rho(t^*)\). When \(\beta>0\), the shear creates rapid radial oscillations at $t=t^*$, and differentiating the phase yields the desired lower bound. When $\beta<0$, the initial density is premixed and the transport unwinds the radial oscillation at $t=t^*$. In this case, duality is used to prove the smallness of the initial density, whereas the final lower bound follows from an $L^p$
lower bound and interpolation with a $\dot{B}^1_{p,\infty}$ upper bound.

To transfer this growth to an exact inviscid solution \((u,\rho)\) with the
same initial data, set
\begin{align*}
	\zeta=\bar u-u,
	\qquad
	\xi=\bar\rho-\rho.
\end{align*}
The resulting perturbation system is forced only by the Boussinesq coupling
omitted from the transport ansatz. Energy estimates show that
\((\zeta,\xi)\) remains small up to \(t^*\), and hence the lower bound for
\(\bar\rho(t^*)\) transfers to \(\rho(t^*)\).

Finally, we pass to the fully dissipative system using the same approximate solution.
The additional terms are \(-\Delta\bar u\) and \(-\Delta\bar\rho\), which are treated
as perturbative errors in the same stability argument. Since they contain two more
derivatives than the inviscid errors, the admissible range changes from
\[
\beta-\frac{2}{p}<1
\qquad\text{to}\qquad
\beta-\frac{2}{p}<-1.
\]
This matches the shift from the inviscid Lipschitz threshold to the critical threshold
determined by the parabolic scaling of the velocity equation.

The paper is organized as follows. In Section \ref{sec:2}, we collect the preliminaries used throughout the paper. More precisely, Subsection \ref{sec:2.1} recalls
the auxiliary analytic tools, while Subsection \ref{sec:2.2} constructs the initial data
and the radial profiles used in the norm inflation argument. Section \ref{sec:3} is
devoted to the inviscid case. In Subsection \ref{sec:3.1}, we introduce the approximate
solution and prove its basic estimates, including the norm inflation of the approximate
density. Subsection \ref{sec:3.2} establishes the stability of the approximation, and
Subsection \ref{sec:3.3} completes the proof of Theorem \ref{thm}. Finally, Section \ref{sec:4} is devoted to the fully dissipative case. Subsection
\ref{sec:4.1} develops the approximation and stability argument in the dissipative
setting, and Subsection \ref{sec:4.2} proves Theorem \ref{thm-1}.

\section{Preliminaries}\label{sec:2}
\subsection{Auxiliary Analysis Tools}\label{sec:2.1}
This subsection collects several standard facts on Sobolev and Besov spaces that will
be used throughout the paper. We begin by fixing our convention for Sobolev spaces.
For an integer \(k\geq0\) and \(1\leq p\leq\infty\), we set
\[
\|f\|_{W^{k,p}}
:=
\sum_{0\leq |\alpha|\leq k}\|\partial^\alpha f\|_{L^p},
\qquad
\|f\|_{\dot W^{k,p}}
:=
\sum_{|\alpha|=k}\|\partial^\alpha f\|_{L^p}.
\]
For fractional or negative regularity indices, we use the Sobolev notation only in the
range \(1<p<\infty\). More precisely, for \(s\in\mathbb R\) and \(1<p<\infty\), we define
\[
\|f\|_{W^{s,p}}
:=
\|(I-\Delta)^{\frac{s}{2}}f\|_{L^p},
\qquad
\|f\|_{\dot W^{s,p}}
:=
\|(-\Delta)^{\frac{s}{2}}f\|_{L^p},
\]
where the operators are defined through the Fourier multipliers
\[
\widehat{(I-\Delta)^{\frac{s}{2}}f}(\xi)
=
(1+|\xi|^2)^{\frac{s}{2}}\widehat f(\xi),
\qquad
\widehat{(-\Delta)^{\frac{s}{2}}f}(\xi)
=
|\xi|^s\widehat f(\xi).
\]
When \(p=2\), we use the standard notation
\[
H^s=W^{s,2},
\qquad
\dot H^s=\dot W^{s,2}.
\]
The endpoint cases \(p=1\) and \(p=\infty\) for fractional Sobolev spaces will not be
used below; at these endpoints we work instead in the Besov scale.

We next recall the homogeneous Besov spaces. Let
\((\dot\Delta_j)_{j\in\mathbb Z}\) be the homogeneous dyadic blocks associated with a
smooth Littlewood--Paley decomposition away from the origin. For \(s\in\mathbb R\) and
\(1\leq p,r\leq\infty\), we define
\[
\|f\|_{\dot B^s_{p,r}}
:=
\left\|
\left(2^{js}\|\dot\Delta_j f\|_{L^p}\right)_{j\in\mathbb Z}
\right\|_{\ell^r}.
\]
The homogeneous Besov space \(\dot B^s_{p,r}(\mathbb R^n)\) consists of tempered
distributions modulo polynomials for which the above norm is finite. The
nonhomogeneous Besov spaces \(B^s_{p,r}\) are defined similarly by using the
nonhomogeneous dyadic blocks. We refer to \cite{b-BCD-b-bahouri-chemin-fourier-book-2011} for the precise
definitions of Besov spaces, the details of the Littlewood--Paley theory, and the
standard properties recalled below, including embeddings, interpolation, duality,
Bernstein inequalities, and Fourier multiplier estimates.

The following embeddings will be used repeatedly.

\begin{lemma}[Basic embeddings]\label{lem-basic-besov-embeddings}
	Let \(s\in\mathbb R\). The following embeddings hold.
	\begin{enumerate}
		\item If \(1\leq p\leq \widetilde p\leq\infty\) and
		\(1\leq r\leq \widetilde r\leq\infty\), then
		\[
		\dot B^s_{p,r}(\mathbb R^n)
		\hookrightarrow
		\dot B^{s-n(\frac1p-\frac1{\widetilde p})}_{\widetilde p,\widetilde r}
		(\mathbb R^n).
		\]
		
		\item For every integer \(k\geq0\) and \(1\leq p\leq\infty\), one has
		\[
		\dot B^k_{p,1}(\mathbb R^n)
		\hookrightarrow
		\dot W^{k,p}(\mathbb R^n)
		\hookrightarrow
		\dot B^k_{p,\infty}(\mathbb R^n).
		\]
		More generally, for \(s\in\mathbb R\) and \(1<p<\infty\),
		\[
		\dot B^s_{p,1}(\mathbb R^n)
		\hookrightarrow
		\dot B^s_{p,\min\{p,2\}}(\mathbb R^n)
		\hookrightarrow
		\dot W^{s,p}(\mathbb R^n)
		\hookrightarrow
		\dot B^s_{p,\max\{p,2\}}(\mathbb R^n)
		\hookrightarrow
		\dot B^s_{p,\infty}(\mathbb R^n).
		\]
	\end{enumerate}
\end{lemma}

The next lemma records the interpolation estimate needed below.
\begin{lemma}[Besov interpolation]\label{lem-besov-interpolation}
	Let \(s_1<s_2\), \(0<\theta<1\), and
	\[
	s=\theta s_1+(1-\theta)s_2.
	\]
	Then, for \(1\leq p\leq\infty\),
	\[
	\|f\|_{\dot B^s_{p,1}}
	\lesssim
	\|f\|_{\dot B^{s_1}_{p,\infty}}^\theta
	\|f\|_{\dot B^{s_2}_{p,\infty}}^{1-\theta}.
	\]
\end{lemma}

We also need the following duality characterization of homogeneous Besov spaces, which will be used to handle negative regularity norms. Set
\[
\mathcal S_0(\mathbb R^n)
:=
\left\{
\phi\in\mathcal S(\mathbb R^n):
\int_{\mathbb R^n}x^\alpha\phi(x)\,dx=0
\text{ for every multi-index }\alpha
\right\}.
\]

\begin{prop}[Besov duality]\label{prop-besov-duality}
	Let \(s\in\mathbb R\) and \(1\leq p,r\leq\infty\). For
	\(f\in\mathcal S'(\mathbb R^n)/\mathcal P(\mathbb R^n)\), one has
	\[
	\|f\|_{\dot B^s_{p,r}}
	\sim
	\sup_{\phi\in Q^{-s}_{p',r'}}
	\left|
	\int_{\mathbb R^n} f(x)\phi(x)\,dx
	\right|,
	\]
	where \(p'\) and \(r'\) denote the H\"older conjugates of \(p\) and \(r\),
	respectively, and
	\[
	Q^{-s}_{p',r'}
	:=
	\left\{
		\phi\in \mathcal S_0(\mathbb R^n)\cap \dot B^{-s}_{p',r'}:
	\|\phi\|_{\dot B^{-s}_{p',r'}}\leq 1
	\right\}.
	\]
\end{prop}

The following Bernstein inequalities and Fourier multiplier estimates will be used throughout the paper.

\begin{lemma}[Bernstein inequalities]\label{lem-bernstein}
	Let \(1\leq p\leq q\leq\infty\), \(k\in\mathbb N_0\), and \(j\in\mathbb Z\).
	Assume that
	\[
	\operatorname{supp}\widehat f
	\subset
	\{\xi\in\mathbb R^n:|\xi|\leq C_0 2^j\}.
	\]
	Then
	\[
	\|\nabla^k f\|_{L^q}
	\lesssim
	2^{jk+nj(\frac1p-\frac1q)}\|f\|_{L^p}.
	\]
	If, in addition,
	\[
	\operatorname{supp}\widehat f
	\subset
	\{\xi\in\mathbb R^n:c_0 2^j\leq |\xi|\leq C_0 2^j\},
	\]
	then
	\[
	\|\nabla^k f\|_{L^p}
	\sim
	2^{jk}\|f\|_{L^p}.
	\]
	The implicit constants are independent of \(j\) and \(f\).
\end{lemma}

The Calder\'on--Zygmund estimate for \(\nabla u\) will be used in \(L^p\) only in the
range \(1<p<\infty\). The endpoint cases \(p=1,\infty\) will be handled in the Besov
scale, where dyadic localization provides boundedness of zero-order Fourier
multipliers for all \(1\leq p,r\leq\infty\).

\begin{lemma}[Biot--Savart and Besov estimates]\label{lem-biot-savart-besov}
	Let \(d=2\), and let \(u\) be a divergence-free vector field with vorticity
	\[
	w=\partial_1u_2-\partial_2u_1.
	\]
	Then
	\[
	\nabla u=\mathcal R w,
	\]
	where \(\mathcal R\) denotes a matrix of zero-order Calder\'on--Zygmund operators.
	Consequently, for \(1<p<\infty\),
	\[
	\|\nabla u\|_{L^p}
	\lesssim
	\|w\|_{L^p}.
	\]
	Moreover, since zero-order Fourier multipliers are bounded on homogeneous Besov
	spaces, for any \(s\in\mathbb R\) and \(1\leq p,r\leq\infty\),
	\[
	\|\nabla u\|_{\dot B^s_{p,r}}
	\lesssim
	\|w\|_{\dot B^s_{p,r}}.
	\]
	Equivalently, since
	\[
	u=\nabla^\perp\Delta^{-1}w,
	\]
	the Biot--Savart operator is a Fourier multiplier of order \(-1\). Hence, for any
	\(s\in\mathbb R\) and \(1\leq p,r\leq\infty\),
	\[
	\|u\|_{\dot B^{s+1}_{p,r}}
	\lesssim
	\|w\|_{\dot B^s_{p,r}}.
	\]
\end{lemma}

\subsection{Construction of Initial Data}\label{sec:2.2}

Taking the curl of the velocity equation in \eqref{boussinesq-orginal-0}, we obtain
the vorticity equation
\begin{align}\label{boussinesq}
	\partial_t w-\mu \Delta w+u\cdot\nabla w=\partial_{1}\rho .
\end{align}
In what follows, we work in polar coordinates \((r,\theta)\), where
\[
x=(r\cos\theta,r\sin\theta).
\]
We denote the associated orthonormal frame by
\[
\mathbf e_r=(\cos\theta,\sin\theta),
\qquad
\mathbf e_\theta=(-\sin\theta,\cos\theta).
\]
For a vector field \(u\), we write
\[
u=u_r\mathbf e_r+u_\theta\mathbf e_\theta,
\qquad
u_r=u\cdot\mathbf e_r,
\qquad
u_\theta=u\cdot\mathbf e_\theta .
\]

We now choose the radial profiles used in the construction. The profile \(f\) will be
used as a radial vorticity profile and generates a stationary angular velocity field,
while \(g\) will be placed in a region where this angular velocity has a nondegenerate
radial shear. The following lemma makes this choice precise; its proof is postponed to
Appendix \ref{app-choice-fg}.

\begin{lemma}\label{lem-choice-fg-inviscid}
There exist one-dimensional profiles \(f,g\in C_c^\infty((0,\infty))\), with
\(f\not\equiv0\) and \(g\not\equiv0\), such that the corresponding radial functions
\(f(r)\) and \(g(r)\) on \(\mathbb R^2\) satisfy
\[
\operatorname{supp}f\cap \operatorname{supp}g=\emptyset,
\qquad
\int_0^\infty f(r)\,dr
=
\int_0^\infty r f(r)\,dr
=0,
\qquad
f^{(-2)}\in C_c^\infty((0,\infty)).
\]
Moreover, if \(u[f]\) denotes the velocity generated by the radial vorticity profile
\(f(r)\), and
\[
h(r)\define \frac{u_\theta[f](r)}{r},
\]
then there exists \(c_0>0\) such that
\[
|h'(r)|\geq c_0
\qquad \text{on } \operatorname{supp}g .
\]
\end{lemma}

With the profiles \(f\) and \(g\) fixed as in Lemma \ref{lem-choice-fg-inviscid}, we
now prescribe the initial data. Let \(\lambda\gg1\) be a large parameter. The initial
vorticity is defined by
\begin{align}\label{initial-vorticity-inviscid}
	w(0,r)
	=
	\lambda^{\frac{2}{p}+1-\beta} (\log\log \lambda)^{-\frac{|\beta|}{2+|\beta|}}f(\lambda r).
\end{align}
The initial density is chosen according to the sign of \(\beta\). If \(\beta>0\), we set
\begin{align}\label{initial-rho-exact=app}
	\rho(0,r,\theta)
	=
	\lambda^{\frac{2}{p}-\beta}
	(\log\log\lambda)^{-\frac{\beta}{2+\beta}}
	g(\lambda r)\cos\theta .
\end{align}
If \(\beta<0\), we set
\begin{align}\label{initial-rho-negative-inviscid}
	\rho(0,r,\theta)
	=
	\lambda^{\frac{2}{p}-\beta-k}
	(\log\log\lambda)^{-\frac{\beta+2k}{2-\beta}}
	\partial_r^k
	\left[
	g(\lambda r)
	\cos\left(
	\theta
	+
	t^*\lambda^{\frac{2}{p}-\beta}(\log\log \lambda)^{\frac{\beta}{2-\beta}}
	\frac{u_\theta[f](\lambda r)}{r}
	\right)
	\right],
\end{align}
where
\begin{align}\label{definition-tstar-k-inviscid}
	t^*
	=
	\lambda^{-\frac{2}{p}-1+\beta}\log\log\lambda,
	\qquad
	k=-\lfloor \beta\rfloor+2.
\end{align}
Since
\[
\frac{u_\theta[f](\lambda r)}{r}
=
\lambda\frac{u_\theta[f](\lambda r)}{\lambda r}
=
\lambda h(\lambda r),
\]
the oscillatory phase in \eqref{initial-rho-negative-inviscid} can equivalently be
written as
\[
\theta+t^*\lambda^{\frac{2}{p}-\beta+1}(\log\log \lambda)^{-\frac{|\beta|}{2+|\beta|}}h(\lambda r).
\]
The nondegeneracy condition in Lemma \ref{lem-choice-fg-inviscid} will be used
later to obtain the lower bound for the density norm at the inflation time.

\section{The Inviscid Case}\label{sec:3}

Throughout this section, we consider the inviscid Boussinesq system, namely
\eqref{boussinesq-orginal-0} with \(\mu=\nu=0\), and prove Theorem
\ref{thm}.

\subsection{The Approximate Solution}\label{sec:3.1}

In this subsection, we construct a family of approximate solutions and derive the
basic estimates needed in the proof of Theorem \ref{thm}.

We begin with the vorticity component. Let \(\lambda\gg1\) be a large parameter. We
define the approximate vorticity by
\begin{align}\label{app-w-expression}
	\bar w(t)=w(0)
	=
	\lambda^{\frac{2}{p}+1-\beta}
	(\log\log \lambda)^{-\frac{|\beta|}{2+|\beta|}}
	f(\lambda r).
\end{align}
Since \(\bar w\) is radial, the velocity \(u[\bar w]\) generated by the Biot--Savart
law,
\begin{align*}
		u[\omega](x)
		=
		\frac{1}{2\pi}
		\int_{\mathbb R^2}
		\frac{(x-y)^\perp \omega(y)}{|x-y|^2}\,dy,
\end{align*}
has no radial component. Hence \(u_r[\bar w]=0\), and therefore
\[
u[\bar w]\cdot\nabla \bar w
=
\frac{1}{r}u_\theta[\bar w]\partial_\theta \bar w
=0.
\]
It follows that \(\bar w\) satisfies
\begin{align*}
	\partial_t\bar w+u[\bar w]\cdot\nabla\bar w=0.
\end{align*}

We next define the approximate density. For \(r>0\), define
\begin{align}\label{definition-angular-frequency}
	\Omega_\lambda(r)
	:=
	\lambda^{\frac{2}{p}-\beta}
	(\log\log\lambda)^{-\frac{|\beta|}{2+|\beta|}}
	\frac{u_\theta[f](\lambda r)}{r}.
\end{align}
Since \(f\) is supported away from the origin, the radial Biot--Savart
formula shows that \(u_\theta[f](r)=0\) for all sufficiently small \(r\).
Therefore, \(\Omega_\lambda\) extends smoothly to \(r=0\), and we set
\[
\Omega_\lambda(0):=0.
\]
Since the velocity generated by \(\bar w\)
has no radial component, the corresponding transport equation can be solved
explicitly:
\begin{align*}
	\begin{cases}
		\partial_t\bar\rho+u[\bar w]\cdot\nabla\bar\rho=0,\\[1mm]
		\bar\rho(0)=\rho(0).
	\end{cases}
\end{align*}
Indeed,
\[
u[\bar w]\cdot\nabla\bar\rho
=
\frac{1}{r}u_\theta[\bar w]\partial_\theta\bar\rho .
\]
Thus the density is transported along angular characteristics.

For \(\beta>0\), we set
\begin{align}\label{app-rho-expression}
	\bar\rho(t)
	=
	\lambda^{\frac{2}{p}-\beta}
	(\log\log \lambda)^{-\frac{\beta}{2+\beta}}
		g(\lambda r)
		\cos \left(
		\theta
		-
		t\Omega_\lambda(r)
		\right).
\end{align}
For \(\beta<0\), we instead use a premixed density:
\begin{align}\label{app-rho-expre-<0}
		\bar\rho_0(r,\vartheta)
		&=
	\lambda^{\frac{2}{p}-\beta-k}
	(\log\log \lambda)^{-\frac{\beta+2k}{2-\beta}}
	\partial_r^k
	\left[
		g(\lambda r)
		\cos \left(
		\vartheta
		+
		t^*\Omega_\lambda(r)
		\right)
		\right],\\
		\bar\rho(t,r,\theta)
		&=
		\bar\rho_0\bigl(r,\theta-t\Omega_\lambda(r)\bigr).
\end{align}
Here $\partial_r^k$ in the definition of $\bar\rho_0$ acts with $\vartheta$ fixed,
\(k=-\lfloor\beta\rfloor+2\), and \(t^*\) is the time at which norm inflation
will occur. 
The different forms of \(\bar\rho\) for positive and negative \(\beta\)
reflect the different mechanisms responsible for the growth: the creation of rapid
radial oscillations in the positive regularity case, and premixing followed by
unmixing in the negative regularity case.

\subsubsection{Estimates for the Approximate Solution}
\label{sec:3.1.1}

Since \(\bar w\) is stationary and its expression is independent of the sign of
\(\beta\), its estimates are more direct than those for \(\bar\rho\). We first treat
the vorticity component.

\begin{lemma}\label{app-w-bound-s geq 0}
	For any \(s>-1\) and \(1\leq q\leq\infty\), there holds
	\begin{align}\label{w-bound-app}
		\|\bar w(t)\|_{\dot B^s_{q,1}}
		\lesssim
		\lambda^{s+1+\frac{2}{p}-\frac{2}{q}-\beta}
		(\log\log \lambda)^{-\frac{|\beta|}{2+|\beta|}},
	\end{align}
	for all \(0\leq t\leq t^*\).
\end{lemma}

\begin{proof}
	We distinguish the sign of the regularity index \(s\).
	
	\medskip
	\noindent{\bf Step 1. The case \(s>0\).}
	From the explicit expression of \(\bar w\), we have
	\begin{align*}
		\|\bar w(t)\|_{L^q}
		\lesssim
		\lambda^{1+\frac{2}{p}-\frac{2}{q}-\beta}
		(\log\log \lambda)^{-\frac{|\beta|}{2+|\beta|}}.
	\end{align*}
	Moreover, for every integer \(k\geq1\),
	\begin{align*}
		\|\bar w(t)\|_{\dot W^{k,q}}
		\lesssim
		\lambda^{k+1+\frac{2}{p}-\frac{2}{q}-\beta}
		(\log\log \lambda)^{-\frac{|\beta|}{2+|\beta|}}.
	\end{align*}
	Let \(m=\lfloor s\rfloor+1\). Then \(m>s\). By Besov interpolation and the
	embeddings
	\[
	L^q\hookrightarrow \dot B^0_{q,\infty},
	\qquad
	\dot W^{m,q}\hookrightarrow \dot B^m_{q,\infty},
	\]
	we obtain
	\begin{align*}
		\|\bar w(t)\|_{\dot B^s_{q,1}}
		&\lesssim
		\|\bar w(t)\|_{\dot B^0_{q,\infty}}^{1-\frac{s}{m}}
		\|\bar w(t)\|_{\dot B^m_{q,\infty}}^{\frac{s}{m}}\\
		&\lesssim
		\|\bar w(t)\|_{L^q}^{1-\frac{s}{m}}
		\|\bar w(t)\|_{\dot W^{m,q}}^{\frac{s}{m}}\\
		&\lesssim
		\lambda^{s+1+\frac{2}{p}-\frac{2}{q}-\beta}
		(\log\log \lambda)^{-\frac{|\beta|}{2+|\beta|}}.
	\end{align*}
	This proves \eqref{w-bound-app} for \(s>0\).
	
	\medskip
	\noindent{\bf Step 2. The case \(-1<s<0\).}
		Let \(\phi\in\mathcal S_0(\mathbb R^2)\). We estimate the duality pairing
	\[
	\int_{\mathbb R^2}\bar w(t,x)\phi(x)\,dx.
	\]
	Using polar coordinates and integrating by parts in \(r\), we get
	\begin{align}\nonumber
		&\int_{\mathbb R^2}\bar w(t,x)\phi(x)\,dx \\\nonumber
		&=
		\lambda^{\frac{2}{p}+1-\beta}
		(\log\log \lambda)^{-\frac{|\beta|}{2+|\beta|}}
		\int_0^\infty\int_{-\pi}^{\pi}
		f(\lambda r)\phi(r,\theta)\,r\,dr\,d\theta \\\nonumber
		&=
		-\lambda^{\frac{2}{p}-\beta}
		(\log\log \lambda)^{-\frac{|\beta|}{2+|\beta|}}
		\int_0^\infty\int_{-\pi}^{\pi}
		f^{(-1)}(\lambda r)
		\partial_r\big(\phi(r,\theta)r\big)\,dr\,d\theta \\\nonumber
		&=
		-\lambda^{\frac{2}{p}-\beta}
		(\log\log \lambda)^{-\frac{|\beta|}{2+|\beta|}}
		\int_0^\infty\int_{-\pi}^{\pi}
		f^{(-1)}(\lambda r)
		\big(r\partial_r\phi(r,\theta)+\phi(r,\theta)\big)\,dr\,d\theta \\\label{J1J2}
		&\define J_1+J_2 .
	\end{align}
	For \(J_1\), we estimate directly:
	\begin{align}\nonumber
		|J_1|
		&\lesssim
		\lambda^{\frac{2}{p}-\beta}
		(\log\log \lambda)^{-\frac{|\beta|}{2+|\beta|}}
		\int_0^\infty\int_{-\pi}^{\pi}
		|f^{(-1)}(\lambda r)|\,r\,|\partial_r\phi|\,dr\,d\theta \\\label{J1}
		&\lesssim
		\lambda^{\frac{2}{p}-\frac{2}{q}-\beta}
		(\log\log \lambda)^{-\frac{|\beta|}{2+|\beta|}}
		\|\phi\|_{\dot B^1_{q',1}}.
	\end{align}
	For \(J_2\), we integrate by parts once more and use the compact support of
	\(f^{(-2)}\):
	\begin{align}\nonumber
		J_2
		&=
		-\lambda^{\frac{2}{p}-\beta}
		(\log\log \lambda)^{-\frac{|\beta|}{2+|\beta|}}
		\int_0^\infty\int_{-\pi}^{\pi}
		f^{(-1)}(\lambda r)\phi(r,\theta)\,dr\,d\theta \\\nonumber
		&=
		\lambda^{\frac{2}{p}-1-\beta}
		(\log\log \lambda)^{-\frac{|\beta|}{2+|\beta|}}
		\int_0^\infty\int_{-\pi}^{\pi}
		\frac1r f^{(-2)}(\lambda r)\partial_r\phi(r,\theta)\,r\,dr\,d\theta .
	\end{align}
	Since \(f^{(-2)}(\lambda r)\) is supported where \(r\sim\lambda^{-1}\), the factor
	\(\frac{1}{r}\) contributes one power of \(\lambda\). Hence
	\begin{align}\label{J2}
		|J_2|
		\lesssim
		\lambda^{\frac{2}{p}-\frac{2}{q}-\beta}
		(\log\log \lambda)^{-\frac{|\beta|}{2+|\beta|}}
		\|\phi\|_{\dot B^1_{q',1}}.
	\end{align}
	Combining \eqref{J1}, \eqref{J2}, and \eqref{J1J2}, and using the Besov duality
	in Proposition \ref{prop-besov-duality}, we obtain
	\begin{align}\label{bar-w-B-1}
		\|\bar w(t)\|_{\dot B^{-1}_{q,\infty}}
		\lesssim
		\lambda^{\frac{2}{p}-\frac{2}{q}-\beta}
		(\log\log \lambda)^{-\frac{|\beta|}{2+|\beta|}}.
	\end{align}
	Since \(L^q\hookrightarrow \dot B^0_{q,\infty}\), interpolation between
	\(\dot B^{-1}_{q,\infty}\) and \(\dot B^0_{q,\infty}\) gives
	\begin{align*}
		\|\bar w(t)\|_{\dot B^s_{q,1}}
		&\lesssim
		\|\bar w(t)\|_{\dot B^{-1}_{q,\infty}}^{-s}
		\|\bar w(t)\|_{\dot B^0_{q,\infty}}^{1+s}\\
		&\lesssim
		\|\bar w(t)\|_{\dot B^{-1}_{q,\infty}}^{-s}
		\|\bar w(t)\|_{L^q}^{1+s}\\
		&\lesssim
		\lambda^{s+1+\frac{2}{p}-\frac{2}{q}-\beta}
		(\log\log \lambda)^{-\frac{|\beta|}{2+|\beta|}}.
	\end{align*}
	This proves \eqref{w-bound-app} for \(-1<s<0\).
	
	\medskip
	\noindent{\bf Step 3. The endpoint \(s=0\).}
	The endpoint is treated separately in order to avoid using noninteger order
	estimates at \(q=1,\infty\). From \eqref{bar-w-B-1},
	\[
	\|\bar w(t)\|_{\dot B^{-1}_{q,\infty}}
	\lesssim
	\lambda^{\frac{2}{p}-\frac{2}{q}-\beta}
	(\log\log \lambda)^{-\frac{|\beta|}{2+|\beta|}}.
	\]
	On the other hand, the integer order estimate from Step 1 with \(k=1\) gives
	\[
	\|\bar w(t)\|_{\dot W^{1,q}}
	\lesssim
	\lambda^{2+\frac{2}{p}-\frac{2}{q}-\beta}
	(\log\log \lambda)^{-\frac{|\beta|}{2+|\beta|}}.
	\]
	Hence, by \(\dot W^{1,q}\hookrightarrow \dot B^1_{q,\infty}\),
	\[
	\|\bar w(t)\|_{\dot B^1_{q,\infty}}
	\lesssim
	\lambda^{2+\frac{2}{p}-\frac{2}{q}-\beta}
	(\log\log \lambda)^{-\frac{|\beta|}{2+|\beta|}}.
	\]
	Interpolating between \(\dot B^{-1}_{q,\infty}\) and \(\dot B^1_{q,\infty}\), we get
	\begin{align*}
		\|\bar w(t)\|_{\dot B^0_{q,1}}
		&\lesssim
		\|\bar w(t)\|_{\dot B^{-1}_{q,\infty}}^{\frac12}
		\|\bar w(t)\|_{\dot B^1_{q,\infty}}^{\frac12}\\
		&\lesssim
		\lambda^{1+\frac{2}{p}-\frac{2}{q}-\beta}
		(\log\log \lambda)^{-\frac{|\beta|}{2+|\beta|}}.
	\end{align*}
	This is \eqref{w-bound-app} when \(s=0\). The proof is complete.
\end{proof}

As an immediate consequence of Lemma \ref{app-w-bound-s geq 0} and the Biot--Savart
estimate in Lemma \ref{lem-biot-savart-besov},
\[
\|u[\bar w](t)\|_{\dot B^s_{q,1}}
\lesssim
\|\bar w(t)\|_{\dot B^{s-1}_{q,1}},
\]
we obtain the following estimate for the approximate velocity.

\begin{coro}\label{app-u-bound-s geq 0}
	For \(s>0\) and \(1\leq q\leq\infty\), there holds
	\begin{align}\label{estimate-bar-u}
		\|u[\bar w](t)\|_{\dot B^s_{q,1}}
		\lesssim
		\lambda^{s+\frac{2}{p}-\frac{2}{q}-\beta}
		(\log\log \lambda)^{-\frac{|\beta|}{2+|\beta|}},
	\end{align}
	for all \(0\leq t\leq t^*\).
\end{coro}

We next estimate the approximate density. Unlike \(\bar w\), the expression of
\(\bar\rho\) depends on the sign of \(\beta\), and the two cases are treated
separately.

\begin{lemma}\label{app-rho-bound-s geq 0}
	For any \(s>0\) and \(1\leq q\leq\infty\), there holds
	\begin{align}\label{estimate-bar-rho}
		\|\bar\rho(t)\|_{\dot B^{s}_{q,1}}
		\lesssim
		\lambda^{s+\frac{2}{p}-\frac{2}{q}-\beta}
		(\log\log\lambda)^{\frac{2s-\beta}{2+|\beta|}},
	\end{align}
	for all \(0\leq t\leq t^*\).
\end{lemma}

\begin{proof}
	It is enough to prove the corresponding integer order estimates and then use Besov
	interpolation. We distinguish the two signs of \(\beta\).
	
	\medskip
	\noindent{\bf Step 1. The case \(\beta>0\).}
	Using
	\[
	\nabla \bar\rho(t)
	=
	\partial_r\bar\rho(t)\mathbf e_r
	+
	\frac1r\partial_\theta\bar\rho(t)\mathbf e_\theta,
	\]
	and the fact that \(g(\lambda r)\) is supported where \(\lambda r\sim1\), we estimate
	the derivatives of \(\bar\rho\) directly. For every integer \(m\geq0\), differentiating
	the explicit formula for \(\bar\rho(t)\) gives
	\begin{align*}
		\|\bar\rho(t)\|_{\dot W^{m,q}}
		&\lesssim
		\lambda^m
		\lambda^{\frac{2}{p}-\frac{2}{q}-\beta}
		(\log\log\lambda)^{-\frac{\beta}{2+\beta}}\\
		&\quad+
		\left(
		t\lambda^{2+\frac{2}{p}-\beta}
		(\log\log\lambda)^{-\frac{\beta}{2+\beta}}
		\right)^m
		\lambda^{\frac{2}{p}-\frac{2}{q}-\beta}
		(\log\log\lambda)^{-\frac{\beta}{2+\beta}}.
	\end{align*}
	Since
	\[
	t\leq t^*
	=
	\lambda^{-\frac{2}{p}-1+\beta}\log\log\lambda,
	\]
	we have
	\[
	t\lambda^{2+\frac{2}{p}-\beta}
	(\log\log\lambda)^{-\frac{\beta}{2+\beta}}
	\lesssim
	\lambda(\log\log\lambda)^{\frac{2}{2+\beta}}.
	\]
	Thus, for every integer \(m\geq0\),
	\begin{align}\label{rho-integer-positive}
		\|\bar\rho(t)\|_{\dot W^{m,q}}
		\lesssim
		\lambda^{m+\frac{2}{p}-\frac{2}{q}-\beta}
		(\log\log\lambda)^{\frac{2m-\beta}{2+\beta}}.
	\end{align}
	
	Now let \(s>0\). Choose an integer \(m=\lfloor s\rfloor+1\), so that \(m>s\).
	By the embeddings
	\[
	L^q\hookrightarrow \dot B^0_{q,\infty},
	\qquad
	\dot W^{m,q}\hookrightarrow \dot B^m_{q,\infty},
	\]
	and Besov interpolation, we obtain
	\begin{align*}
		\|\bar\rho(t)\|_{\dot B^s_{q,1}}
		&\lesssim
		\|\bar\rho(t)\|_{L^q}^{1-\frac{s}{m}}
		\|\bar\rho(t)\|_{\dot W^{m,q}}^{\frac{s}{m}}\\
		&\lesssim
		\left(
		\lambda^{\frac{2}{p}-\frac{2}{q}-\beta}
		(\log\log\lambda)^{-\frac{\beta}{2+\beta}}
		\right)^{1-\frac{s}{m}}
		\left(
		\lambda^{m+\frac{2}{p}-\frac{2}{q}-\beta}
		(\log\log\lambda)^{\frac{2m-\beta}{2+\beta}}
		\right)^{\frac{s}{m}}\\
		&\lesssim
		\lambda^{s+\frac{2}{p}-\frac{2}{q}-\beta}
		(\log\log\lambda)^{\frac{2s-\beta}{2+\beta}}.
	\end{align*}
	This proves \eqref{estimate-bar-rho} when \(\beta>0\).
	
	\medskip
	\noindent{\bf Step 2. The case \(\beta<0\).}
	We now consider the premixed density. Expanding the \(k\)-th radial derivative, we
	have
	\begin{align*}
		&\partial_r^k
		\left[
		g(\lambda r)
		\cos\left(
		\theta
		+
		t^*\lambda^{\frac{2}{p}-\beta}
		(\log\log\lambda)^{\frac{\beta}{2-\beta}}
		\frac{u_\theta[f](\lambda r)}{r}
		\right)
		\right]\\
		&=
		\sum_{0\leq i\leq k}
		\lambda^{k-i}g^{(k-i)}(\lambda r)
		\partial_r^i
		\cos\left(
		\theta
		+
		t^*\lambda^{\frac{2}{p}-\beta+1}
		(\log\log\lambda)^{\frac{\beta}{2-\beta}}
		\frac{u_\theta[f](\lambda r)}{\lambda r}
		\right).
	\end{align*}
	Set
	\[
	\Theta
	=
	\theta
	+
	t^*\lambda^{\frac{2}{p}-\beta+1}
	(\log\log\lambda)^{\frac{\beta}{2-\beta}}
	\frac{u_\theta[f](\lambda r)}{\lambda r},
	\qquad
	h(r)=\frac{u_\theta[f](r)}{r}.
	\]
	By the Faà di Bruno formula,
	\begin{align*}
		&\partial_r^i
		\cos\left(
		\theta
		+
		t^*\lambda^{\frac{2}{p}-\beta+1}
		(\log\log\lambda)^{\frac{\beta}{2-\beta}}
		\frac{u_\theta[f](\lambda r)}{\lambda r}
		\right)\\
		&=
		\sum_{m_1+2m_2+\cdots+im_i=i}
		C_{m_1,\cdots,m_i}
		\cos^{(m_1+\cdots+m_i)}(\Theta)\\
		&\quad\times
		\prod_{1\leq j\leq i}
		\left[
		\partial_r^j
		\left(
		t^*\lambda^{\frac{2}{p}-\beta+1}
		(\log\log\lambda)^{\frac{\beta}{2-\beta}}
		h(\lambda r)
		\right)
		\right]^{m_j}.
	\end{align*}
	Since \(t^*=\lambda^{-\frac{2}{p}-1+\beta}\log\log\lambda\), on the support of
	\(g(\lambda r)\),
	\[
		\partial_r^j
		\left(
		t^*\lambda^{\frac{2}{p}-\beta+1}
		(\log\log\lambda)^{\frac{\beta}{2-\beta}}
		h(\lambda r)
		\right)
		=
		\lambda^j(\log\log\lambda)^{\frac{2}{2-\beta}}
		h^{(j)}(\lambda r).
	\]
	Substituting this expansion into \(\bar\rho(t)\) and using the transport structure,
	we write
	\begin{align}\nonumber
		\bar\rho(t)
		&=
		\lambda^{\frac{2}{p}-\beta-k}
		(\log\log\lambda)^{-\frac{\beta+2k}{2-\beta}}
		\sum_{0\leq i\leq k}
		\lambda^{k-i}g^{(k-i)}(\lambda r)
		\sum_{m_1+2m_2+\cdots+im_i=i}
		C_{m_1,\cdots,m_i} \\\nonumber
		&\quad\times
		\cos^{(m_1+\cdots+m_i)}
		\left(
		\theta
		+
		(t^*-t)\lambda^{\frac{2}{p}-\beta+1}
		(\log\log\lambda)^{\frac{\beta}{2-\beta}}
		h(\lambda r)
		\right) \\\nonumber
		&\quad\times
		\prod_{1\leq j\leq i}
		\left(\lambda^j(\log\log\lambda)^{\frac{2}{2-\beta}}\right)^{m_j}
		\left(h^{(j)}(\lambda r)\right)^{m_j} \\\label{app-rho-<0-t}
		&=
		\bar\rho_{\rm lead}(t)+\bar\rho_{\rm low}(t).
	\end{align}
	The leading term corresponds to \(i=k\), \(m_1=k\), and \(m_j=0\) for \(j\geq2\):
	\begin{align*}
		\bar\rho_{\rm lead}(t)
		&=
		\lambda^{\frac{2}{p}-\beta}
		(\log\log\lambda)^{-\frac{\beta}{2-\beta}}
		g(\lambda r)
		\left(h'(\lambda r)\right)^k\\
		&\quad\times
		\cos^{(k)}
		\left(
		\theta
		+
		(t^*-t)\lambda^{\frac{2}{p}-\beta+1}
		(\log\log\lambda)^{\frac{\beta}{2-\beta}}
		h(\lambda r)
		\right).
	\end{align*}
		All remaining terms contain at least one fewer factor
		$(\log\log\lambda)^{\frac{2}{2-\beta}}$, and therefore
		\[
		|\bar\rho_{\rm low}(t)|
		\lesssim
		\lambda^{\frac{2}{p}-\beta}
		(\log\log\lambda)^{-\frac{\beta+2}{2-\beta}}.
	\]
	Arguing as in Step 1, using again that the support lies where \(\lambda r\sim1\), we
	obtain, for every integer \(m\geq0\),
	\begin{align}\label{rho-integer-negative}
		\|\bar\rho(t)\|_{\dot W^{m,q}}
		\lesssim
		\lambda^{m+\frac{2}{p}-\frac{2}{q}-\beta}
		(\log\log\lambda)^{\frac{2m-\beta}{2-\beta}}.
	\end{align}
	Indeed, the leading term satisfies the right-hand side of
		\eqref{rho-integer-negative}, while the lower-order part is smaller by a factor
		$(\log\log\lambda)^{-\frac{2}{2-\beta}}$.
	
	Finally, let \(s>0\) and choose \(m=\lfloor s\rfloor+1\). Interpolating between the
	\(L^q\) estimate and the \(\dot W^{m,q}\) estimate as in Step 1 yields
	\[
	\|\bar\rho(t)\|_{\dot B^s_{q,1}}
	\lesssim
	\lambda^{s+\frac{2}{p}-\frac{2}{q}-\beta}
	(\log\log\lambda)^{\frac{2s-\beta}{2-\beta}}.
	\]
	Since \(2-\beta=2+|\beta|\) when \(\beta<0\), this is exactly
	\eqref{estimate-bar-rho}. The proof is complete.
\end{proof}

\subsubsection{Norm Inflation of the Approximate Solution}\label{sec:3.1.2}

We now establish norm inflation for the approximate solution in the inviscid case. We begin with the velocity component, which remains bounded in the inflation norm and therefore does not contribute to the logarithmic growth.

\begin{lemma}\label{lem-bar-u-beta-estimate}
	Let \(\beta\) and \(p\) satisfy the assumptions of Theorem \ref{thm}. Then the
	approximate velocity \(\bar u=u[\bar w]\) satisfies
	\begin{align}\label{bar-u-beta-estimate}
		\|\bar u(t)\|_{\dot B^\beta_{p,1}}
		=
		\|\bar u(0)\|_{\dot B^\beta_{p,1}}
		\lesssim (\log\log \lambda)^{-\frac{|\beta|}{2+|\beta|}},
	\end{align}
	for \(0\leq t\leq t^*\).
\end{lemma}
\begin{proof}
		Since \(\bar w\) is stationary, so is \(\bar u=u[\bar w]\). By the assumptions
		\(\beta-\frac{2}{p}>-2\) and $p>1$, we may choose \(0<\delta<1\) so small that
		\[
		s\define \beta-\frac{2}{p}+2-\delta>0,
		\qquad
		q\define \frac{2}{2-\delta}\leq p.
	\]
	Then \(s-\frac{2}{q}=\beta-\frac{2}{p}\), and hence the Besov embedding gives
	\[
	\dot B^s_{q,1}
	\hookrightarrow
	\dot B^\beta_{p,1}.
	\]
	Using Corollary \ref{app-u-bound-s geq 0}, we obtain
	\[
	\|\bar u(t)\|_{\dot B^\beta_{p,1}}
	\lesssim
	\|\bar u(t)\|_{\dot B^s_{q,1}}
	\lesssim
	\lambda^{s+\frac{2}{p}-\frac{2}{q}-\beta}
	(\log\log \lambda)^{-\frac{|\beta|}{2+|\beta|}}
	=
	(\log\log \lambda)^{-\frac{|\beta|}{2+|\beta|}}.
	\]
	Together with the stationarity of \(\bar u\), this proves
	\eqref{bar-u-beta-estimate}.
\end{proof}

We next turn to the density component. Although the initial density is small in
\(\dot B^\beta_{p,1}\), the stationary angular transport produces the desired
growth at \(t=t^*\): it creates rapid radial oscillations when \(\beta>0\), whereas
it unwinds the premixed radial oscillation when \(\beta<0\). These two mechanisms
yield the lower bound in \(\dot B^\beta_{p,\infty}\) needed for norm inflation.
Accordingly, the proof is divided according to the sign of \(\beta\).
\begin{lemma}\label{lem-rho-lower-upper-bound}
	Let \(\beta\) and \(p\) satisfy the assumptions of Theorem \ref{thm}. Then
	\[
	\rho(0)=\bar\rho(0),
	\]
	and the common initial density satisfies
	\begin{align}\label{rho-initial-small-besov}
		\|\rho(0)\|_{\dot B^\beta_{p,1}}
		=
		\|\bar\rho(0)\|_{\dot B^\beta_{p,1}}
		\lesssim
		(\log\log \lambda)^{-\frac{|\beta|}{2+|\beta|}}.
	\end{align}
	Moreover, at \(t=t^*\),
	\begin{align}\label{rho-final-large-besov}
		\|\bar\rho(t^*)\|_{\dot B^\beta_{p,\infty}}
		\gtrsim
		(\log\log \lambda)^{\frac{|\beta|}{2+|\beta|}}.
	\end{align}
\end{lemma}
\begin{proof}
	We split the proof according to the sign of \(\beta\).
	
	\medskip
	\noindent{\bf Step 1. The case \(\beta>0\).}
	Let
	\[
	\ell=\lfloor\beta\rfloor+1.
	\]
	Then \(\ell>\beta\). By the definition \eqref{initial-rho-exact=app} and interpolation,
	we have
	\begin{align*}
		\|\rho(0)\|_{\dot B^\beta_{p,1}}
		&\lesssim
		\|\rho(0)\|_{L^p}^{1-\frac{\beta}{\ell}}
		\|\rho(0)\|_{\dot W^{\ell,p}}^{\frac{\beta}{\ell}}\\
		&\lesssim
		\left(
		\lambda^{-\beta}
		(\log\log\lambda)^{-\frac{\beta}{2+\beta}}
		\right)^{1-\frac{\beta}{\ell}}
		\left(
		\lambda^{\ell-\beta}
		(\log\log\lambda)^{-\frac{\beta}{2+\beta}}
		\right)^{\frac{\beta}{\ell}}\\
		&\lesssim
		(\log\log\lambda)^{-\frac{\beta}{2+\beta}}.
	\end{align*}
	This gives the desired smallness of the initial density.
	
	We now evaluate the approximate density at \(t=t^*\). From
	\eqref{app-rho-expression},
	\begin{align*}
		\bar\rho(t^*)
		&=
		\lambda^{\frac{2}{p}-\beta}
		(\log\log\lambda)^{-\frac{\beta}{2+\beta}}
		g(\lambda r)
		\cos\left(
		\theta
		-
		(\log\log\lambda)^{\frac{2}{2+\beta}}
		\frac{u_\theta[f](\lambda r)}{\lambda r}
		\right).
	\end{align*}
	Since
	\[
	\frac{u_\theta[f](\lambda r)}{\lambda r}=h(\lambda r),
	\]
	and \(|h'|\geq c_0\) on \(\operatorname{supp}g\), the leading contribution to
	\(\nabla^\ell\bar\rho(t^*)\) comes from differentiating the oscillatory phase
	\(\ell\) times. More precisely, the leading term has size
	\[
	\lambda^{\frac{2}{p}-\beta}
	(\log\log\lambda)^{-\frac{\beta}{2+\beta}}
	\left(
	\lambda(\log\log\lambda)^{\frac{2}{2+\beta}}
	\right)^\ell,
	\]
		whereas all lower-order terms contain at most \(\ell-1\) derivatives falling on the
	phase and are smaller by a factor of
	\((\log\log\lambda)^{-\frac{2}{2+\beta}}\). Hence
	\begin{align}\label{rho-Well-lower-beta-positive}
		\|\bar\rho(t^*)\|_{\dot W^{\ell,p}}
		\gtrsim
		\lambda^{\ell-\beta}
		(\log\log\lambda)^{\frac{2\ell-\beta}{2+\beta}}.
	\end{align}
	On the other hand, Lemma \ref{app-rho-bound-s geq 0} gives
	\begin{align}\label{rho-Wellplus1-upper-beta-positive}
		\|\bar\rho(t^*)\|_{\dot W^{\ell+1,p}}
		\lesssim
		\lambda^{\ell+1-\beta}
		(\log\log\lambda)^{\frac{2(\ell+1)-\beta}{2+\beta}}.
	\end{align}
	Interpolating between \(\dot B^\beta_{p,\infty}\) and
	\(\dot W^{\ell+1,p}\), we obtain
	\[
	\|\bar\rho(t^*)\|_{\dot W^{\ell,p}}
	\lesssim
	\|\bar\rho(t^*)\|_{\dot B^\beta_{p,\infty}}^{\frac{1}{\ell+1-\beta}}
	\|\bar\rho(t^*)\|_{\dot W^{\ell+1,p}}^{\frac{\ell-\beta}{\ell+1-\beta}}.
	\]
	Combining this with \eqref{rho-Well-lower-beta-positive} and
	\eqref{rho-Wellplus1-upper-beta-positive}, we get
	\[
	\|\bar\rho(t^*)\|_{\dot B^\beta_{p,\infty}}
	\gtrsim
	(\log\log\lambda)^{\frac{\beta}{2+\beta}}.
	\]
	This proves the lemma for \(\beta>0\).
	
	\medskip
	\noindent{\bf Step 2. The case \(\beta<0\).}
	We recall from \eqref{initial-rho-negative-inviscid} that
	\begin{align}\label{initial-rho-<0}
		\bar\rho(0)=\rho(0)
		=
		\lambda^{\frac{2}{p}-\beta-k}
		(\log\log\lambda)^{-\frac{\beta+2k}{2-\beta}}
		\partial_r^k
		\left[
		g(\lambda r)
		\cos \left(
		\theta+
		t^*\lambda^{\frac{2}{p}-\beta}
		(\log\log\lambda)^{\frac{\beta}{2-\beta}}
		\frac{u_\theta[f](\lambda r)}{r}
		\right)
		\right].
	\end{align}
	
We estimate the initial norm by duality. Let
	\(\phi\in\mathcal S_0(\mathbb R^2)\). Then
\begin{align*}
	&\int_{\mathbb R^2}\rho(0)\phi\,dx\\
	&=
	\lambda^{\frac{2}{p}-\beta-k}
	(\log\log\lambda)^{-\frac{\beta+2k}{2-\beta}}\\
	&\quad\times
	\int_0^\infty\int_{-\pi}^{\pi}
	\partial_r^k
	\left[
	g(\lambda r)
	\cos\left(
	\theta+
	t^*\lambda^{\frac{2}{p}-\beta}
	(\log\log\lambda)^{\frac{\beta}{2-\beta}}
	\frac{u_\theta[f](\lambda r)}{r}
	\right)
	\right]\phi\,r\,dr\,d\theta.
\end{align*}
Integrating by parts \(k-1\) times gives
\begin{align}\nonumber
	&\int_0^\infty\int_{-\pi}^{\pi}
	\partial_r^k
	\left[
	g(\lambda r)
	\cos\left(
	\theta+
	t^*\lambda^{\frac{2}{p}-\beta}
	(\log\log\lambda)^{\frac{\beta}{2-\beta}}
	\frac{u_\theta[f](\lambda r)}{r}
	\right)
	\right]\phi\,r\,dr\,d\theta\\
	&=
	(-1)^{k-1}
	\int_0^\infty\int_{-\pi}^{\pi}
	\partial_r
	\left[
	g(\lambda r)
	\cos\left(
	\theta+
	t^*\lambda^{\frac{2}{p}-\beta}
	(\log\log\lambda)^{\frac{\beta}{2-\beta}}
	\frac{u_\theta[f](\lambda r)}{r}
	\right)
	\right]\nonumber\\
	&\qquad\qquad\qquad
	\times
	\left((k-1)\partial_r^{k-2}\phi+r\partial_r^{k-1}\phi\right)
	\,dr\,d\theta \nonumber\\
	&\define I_1+I_2. \label{I1+I2}
\end{align}
For \(I_1\), another integration by parts gives
\begin{align}\nonumber
	I_1
	&=
	(-1)^k
	(k-1)
	\int_0^\infty\int_{-\pi}^{\pi}
	g(\lambda r)
	\cos\left(
	\theta+
	t^*\lambda^{\frac{2}{p}-\beta}
	(\log\log\lambda)^{\frac{\beta}{2-\beta}}
	\frac{u_\theta[f](\lambda r)}{r}
	\right)
	\partial_r^{k-1}\phi\,dr\,d\theta \nonumber\\
	&=
	(-1)^k
	(k-1)
	\int_0^\infty\int_{-\pi}^{\pi}
	\frac1r
	g(\lambda r)
	\cos\left(
	\theta+
	t^*\lambda^{\frac{2}{p}-\beta}
	(\log\log\lambda)^{\frac{\beta}{2-\beta}}
	\frac{u_\theta[f](\lambda r)}{r}
	\right)
	\partial_r^{k-1}\phi\,r\,dr\,d\theta \nonumber\\
	&\lesssim
	\lambda^{1-\frac{2}{p}}
	\|\phi\|_{\dot B^{k-1}_{p',1}}.
	\label{I1}
\end{align}
For \(I_2\), using
\[
t^*\lambda^{\frac{2}{p}-\beta}
(\log\log\lambda)^{\frac{\beta}{2-\beta}}
\frac{u_\theta[f](\lambda r)}{r}
=
(\log\log\lambda)^{\frac{2}{2-\beta}}
\frac{u_\theta[f](\lambda r)}{\lambda r},
\]
we obtain
\begin{align}\nonumber
	I_2
	&=
	(-1)^{k-1}
	\int_0^\infty\int_{-\pi}^{\pi}
	\partial_r
	\left[
	g(\lambda r)
	\cos\left(
	\theta+
	(\log\log\lambda)^{\frac{2}{2-\beta}}
	\frac{u_\theta[f](\lambda r)}{\lambda r}
	\right)
	\right]
	\partial_r^{k-1}\phi\,r\,dr\,d\theta \nonumber\\
	&\lesssim
	\lambda^{1-\frac{2}{p}}
	(\log\log\lambda)^{\frac{2}{2-\beta}}
	\|\phi\|_{\dot B^{k-1}_{p',1}}.
	\label{I2}
\end{align}	

	Combining \eqref{I1+I2}, \eqref{I1}, and \eqref{I2}, we find
\begin{align*}
	\left|
	\int_{\mathbb R^2}\rho(0)\phi\,dx
	\right|
	&\lesssim
	\lambda^{\frac{2}{p}-\beta-k}
	(\log\log\lambda)^{-\frac{\beta+2k}{2-\beta}}
	\lambda^{1-\frac{2}{p}}
	(\log\log\lambda)^{\frac{2}{2-\beta}}
	\|\phi\|_{\dot B^{k-1}_{p',1}}\\
	&\lesssim
	\lambda^{1-\beta-k}
	(\log\log\lambda)^{\frac{2-\beta-2k}{2-\beta}}
	\|\phi\|_{\dot B^{k-1}_{p',1}}.
\end{align*}
By the duality characterization of Besov spaces,
\[
\|\rho(0)\|_{\dot B^{1-k}_{p,\infty}}
\lesssim
\lambda^{1-\beta-k}
(\log\log\lambda)^{\frac{2-\beta-2k}{2-\beta}}.
\]
Moreover, directly from \eqref{initial-rho-<0},
\[
\|\rho(0)\|_{L^p}
\lesssim
\lambda^{-\beta}
(\log\log\lambda)^{-\frac{\beta}{2-\beta}}.
\]
Since \(k=-\lfloor\beta\rfloor+2\), we have \(1-k<\beta<0\). Interpolation between
\(\dot B^{1-k}_{p,\infty}\) and \(L^p\hookrightarrow \dot B^0_{p,\infty}\) yields
\begin{align*}
	\|\rho(0)\|_{\dot B^\beta_{p,1}}
	&\lesssim
	\|\rho(0)\|_{L^p}^{\frac{k-1+\beta}{k-1}}
	\|\rho(0)\|_{\dot B^{1-k}_{p,\infty}}^{\frac{-\beta}{k-1}}\\
	&\lesssim
	(\log\log\lambda)^{\frac{\beta}{2-\beta}}
	=
	(\log\log\lambda)^{-\frac{|\beta|}{2+|\beta|}}.
\end{align*}

		It remains to prove the lower bound at \(t=t^*\). Evaluating
	\eqref{app-rho-<0-t} at \(t=t^*\), we write
	\begin{align*}
		\bar\rho(t^*)
		&=
		\lambda^{\frac{2}{p}-\beta-k}
		(\log\log\lambda)^{-\frac{\beta+2k}{2-\beta}}
		\sum_{0\leq i\leq k}
		\lambda^{k-i} g^{(k-i)}(\lambda r)
		\sum_{m_1+2m_2+\cdots+im_i=i}
		C_{m_1,\cdots,m_i}\\
		&\quad\times
		\cos^{(m_1+\cdots+m_i)}(\theta)
		\prod_{1\leq j\leq i}
		\left(\lambda^j(\log\log\lambda)^{\frac{2}{2-\beta}}\right)^{m_j}
		\left(h^{(j)}(\lambda r)\right)^{m_j}\\
		&=
		\bar\rho_{\mathrm{lead}}(t^*)
		+
		\bar\rho_{\mathrm{low}}(t^*),
	\end{align*}
	where
	\[
	\bar\rho_{\mathrm{lead}}(t^*)
	=
	\lambda^{\frac{2}{p}-\beta}
	(\log\log\lambda)^{-\frac{\beta}{2-\beta}}
	g(\lambda r)
	\cos^{(k)}(\theta)
	\left(h'(\lambda r)\right)^k.
	\]
		As in the proof of Lemma \ref{app-rho-bound-s geq 0}, the lower-order part is
		smaller by a factor $(\log\log\lambda)^{-\frac{2}{2-\beta}}$. Hence the leading
		term yields
	\[
	\|\bar\rho(t^*)\|_{L^p}
	\gtrsim
	\lambda^{-\beta}
	(\log\log\lambda)^{-\frac{\beta}{2-\beta}}.
	\]
	At \(t=t^*\), the leading profile no longer contains the radial oscillatory phase.
	Therefore, using the preceding expansion directly, we also have the sharper bound
	\[
	\|\bar\rho(t^*)\|_{\dot B^1_{p,\infty}}
	\lesssim
	\lambda^{1-\beta}
	(\log\log\lambda)^{-\frac{\beta}{2-\beta}}.
	\]
	By interpolation,
	\[
	\|\bar\rho(t^*)\|_{L^p}
	\lesssim
	\|\bar\rho(t^*)\|_{\dot B^\beta_{p,\infty}}^{\frac{1}{1-\beta}}
	\|\bar\rho(t^*)\|_{\dot B^1_{p,\infty}}^{\frac{-\beta}{1-\beta}}.
	\]
	Consequently,
	\[
	\|\bar\rho(t^*)\|_{\dot B^\beta_{p,\infty}}
	\gtrsim
	(\log\log\lambda)^{-\frac{\beta}{2-\beta}}
	=
	(\log\log\lambda)^{\frac{|\beta|}{2+|\beta|}}.
	\]
	This completes the proof of the lemma.
\end{proof}

\subsection{Stability of Approximation}\label{sec:3.2}
We now compare the approximate solution with the exact inviscid Boussinesq solution
with the same initial data. By construction, the approximate pair
\((\bar w,\bar\rho)\) satisfies
\[
\partial_t\bar w+u[\bar w]\cdot\nabla\bar w=0,
\qquad
\partial_t\bar\rho+u[\bar w]\cdot\nabla\bar\rho=0.
\]
Writing \(\bar u=u[\bar w]\), we may equivalently express this system in velocity
form as
\begin{align*}
	\begin{cases}
		\partial_t \bar u+\bar u\cdot\nabla \bar u+\nabla \bar p
		=
		\bar\rho e_2+G,\\[1mm]
		\partial_t \bar\rho+\bar u\cdot\nabla\bar\rho=0,
	\end{cases}
\end{align*}
where
\[
G=-\bar\rho e_2.
\]
Indeed, the velocity equation for \(\bar u\) is generated only by the stationary
Euler dynamics of \(\bar w\), while the Boussinesq forcing \(\bar\rho e_2\) is not
included in the approximate velocity evolution and is therefore treated as an error.
Combining \eqref{rho-integer-positive} and
\eqref{rho-integer-negative}, we obtain, for every integer
\(s\geq 0\) and \(1\leq q\leq\infty\),
\begin{align}\label{estimate-G-inviscid}
	\|G(t)\|_{\dot W^{s,q}}
	\leq
	\|\bar\rho(t)\|_{\dot W^{s,q}}
	\lesssim
	\lambda^{s+\frac{2}{p}-\frac{2}{q}-\beta}
	(\log\log\lambda)^{\frac{2s-\beta}{2+|\beta|}},
\end{align}
for all \(0\leq t\leq t^*\).

Let \((u,\rho)\) be the exact solution of the inviscid Boussinesq system with the
same initial data:
\begin{align*}
	\begin{cases}
		\partial_t u+u\cdot\nabla u+\nabla p=\rho e_2,\\[1mm]
		\partial_t \rho+u\cdot\nabla\rho=0.
	\end{cases}
\end{align*}
We introduce the error variables
\[
\zeta\define \bar u-u,
\qquad
\xi\define \bar\rho-\rho .
\]
Subtracting the exact system from the approximate one gives
\begin{align}\label{pertur-system-inviscid}
	\begin{cases}
		\partial_t \zeta+\bar u\cdot\nabla \zeta-\zeta\cdot\nabla\zeta
		+\nabla(\bar p-p)
		=
		-\zeta\cdot\nabla\bar u+\xi e_2+G,\\[1mm]
		\partial_t \xi+\bar u\cdot\nabla \xi-\zeta\cdot\nabla\xi
		=
		-\zeta\cdot\nabla\bar\rho .
	\end{cases}
\end{align}
The pressure difference is determined by
\begin{align*}
	\bar p-p
	&=
	(-\Delta)^{-1}\diver\left(
	\bar u\cdot\nabla\zeta
	+\zeta\cdot\nabla\bar u
	-\zeta\cdot\nabla\zeta
	-\xi e_2-G
	\right).
\end{align*}
Since the two solutions have the same initial data,
\[
\zeta(0)=0,
\qquad
\xi(0)=0.
\]

The next lemma shows that the exact solution remains close to the approximate one
up to time \(t^*\).
For convenience, set $L=\log\log\lambda$. In the stability estimates below,
$\mathcal E_\lambda$ denotes a factor bounded by
\[
\mathcal E_\lambda\leq \exp(CL^2)=\lambda^{o(1)},
\]
where $C$ is independent of $\lambda$ and may increase from line to line.
\begin{lemma}\label{lem-error-estimate-inviscid}
	Let \(1<q\leq\infty\) and \(k\in\mathbb N_0\). Then, for all
	\[
	0<t\leq t^*
	=
	\lambda^{-\frac{2}{p}-1+\beta}\log\log\lambda,
	\]
	one has
	\begin{align}\label{error-estimate-short-time-inviscid}
		\|\nabla^k \xi(t),\nabla^k \zeta(t)\|_{L^q}
		\leq
		C\lambda^{k-1-\frac{2}{q}}
		\mathcal E_\lambda.
	\end{align}
\end{lemma}

\begin{proof}
	We first prove the estimates in \(L^2\), and then pass to general \(L^q\) by
	interpolation, together with a low-integrability estimate.
	
	\medskip
	\noindent{\bf Step 1. The \(L^2\) estimate.}
	Taking the \(L^2\) energy estimate for \eqref{pertur-system-inviscid} and using
	\(\operatorname{div}u=\operatorname{div}\bar u=0\), we obtain
	\begin{align}\nonumber
		\frac{d}{dt}\|\zeta,\xi\|_{L^2}
		&\lesssim
		\|\xi\|_{L^2}
		+\|G\|_{L^2}
		+\|\zeta\cdot\nabla\bar u\|_{L^2}
		+\|\zeta\cdot\nabla\bar\rho\|_{L^2} \nonumber\\
		&\lesssim
		\|\xi\|_{L^2}
		+\lambda^{\frac{2}{p}-1-\beta}
		(\log\log\lambda)^{\frac{-\beta}{2+|\beta|}} \nonumber\\
		&\quad
		+\|\zeta\|_{L^2}
		\lambda^{1+\frac{2}{p}-\beta}
		\left(
		\left(\log\log\lambda\right)^{-\frac{|\beta|}{2+|\beta|}}+	\left(\log\log\lambda\right)^{\frac{2-\beta}{2+|\beta|}}
		\right).
		\label{L2-norm-energy-inviscid}
	\end{align}
	Here we used \eqref{estimate-bar-u}, \eqref{estimate-bar-rho}, and
	\eqref{estimate-G-inviscid}.
		Since $0<t\leq t^*$ and
		\[
		1+\frac{2-\beta}{2+|\beta|}\leq2,
		\]
		we have
		\begin{align}
		&t\lambda^{1+\frac{2}{p}-\beta}
		\left(1+L^{-\frac{|\beta|}{2+|\beta|}}
		+L^{\frac{2-\beta}{2+|\beta|}}\right)
		\lesssim L+L^2, \label{exp-gronwall-bound-1-inviscid}\\
		&\exp\left[
		Ct\lambda^{1+\frac{2}{p}-\beta}
		\left(1+L^{-\frac{|\beta|}{2+|\beta|}}
		+L^{\frac{2-\beta}{2+|\beta|}}\right)
		\right]
		\leq \mathcal E_\lambda.\notag
		\end{align}
		Thus Gronwall's inequality gives
	\begin{align}\nonumber
		\|\zeta,\xi\|_{L^2}
		&\lesssim
		t\lambda^{\frac{2}{p}-1-\beta}
		(\log\log\lambda)^{\frac{-\beta}{2+|\beta|}}
		\mathcal E_\lambda \nonumber\\
		&\lesssim
		\lambda^{-2}
		\mathcal E_\lambda.
		\label{L2norm-error-inviscid}
	\end{align}
	This proves \eqref{error-estimate-short-time-inviscid} for \(k=0\) and \(q=2\).
	
	\medskip
	\noindent{\bf Step 2. Higher-order \(L^2\) estimates.}
	We proceed by induction on \(k\). Assume that the estimate has been proved for
	all integers \(0\leq k\leq k_0-1\), and consider \(k=k_0\). We impose the bootstrap
	assumption
		\begin{align}\label{bootstrap-assumption-inviscid}
			\|\nabla\zeta,\nabla\xi\|_{L^\infty}
			\leq
			\exp(C_0L^2),
		\end{align}
		where \(C_0\) will be fixed at the end.
	
	Applying \(\nabla^{k_0}\) to \eqref{pertur-system-inviscid} and taking the
	\(L^2\) energy estimate, we get
	\begin{align}\nonumber
		\frac{d}{dt}
		\|\nabla^{k_0}\zeta,\nabla^{k_0}\xi\|_{L^2}
		&\lesssim
		\|\nabla^{k_0}\xi\|_{L^2}
		+\|\nabla^{k_0}G\|_{L^2} \nonumber\\
		&\quad
		+\|\nabla^{k_0}(\zeta\cdot\nabla\bar u)\|_{L^2}
		+\|\nabla^{k_0}(\zeta\cdot\nabla\bar\rho)\|_{L^2} \nonumber\\
		&\quad
		+\left\|
		\nabla^{k_0}\big((\bar u-\zeta)\cdot\nabla\zeta\big)
		-(\bar u-\zeta)\cdot\nabla\nabla^{k_0}\zeta
		\right\|_{L^2} \nonumber\\
		&\quad
		+\left\|
		\nabla^{k_0}\big((\bar u-\zeta)\cdot\nabla\xi\big)
		-(\bar u-\zeta)\cdot\nabla\nabla^{k_0}\xi
		\right\|_{L^2}.
		\label{error-L2-norm-inviscid}
	\end{align}
	By \eqref{estimate-G-inviscid},
	\begin{align}\label{G-k0-estimate-inviscid}
		\|\nabla^{k_0}G\|_{L^2}
		\lesssim
		\lambda^{k_0-1+\frac{2}{p}-\beta}
		(\log\log\lambda)^{\frac{2k_0-\beta}{2+|\beta|}}.
	\end{align}
	Moreover, using \eqref{estimate-bar-u}, \eqref{estimate-bar-rho}, and
	\eqref{L2norm-error-inviscid},
	\begin{align}\nonumber
		&\|\nabla^{k_0}(\zeta\cdot\nabla\bar u)\|_{L^2}
		+
		\|\nabla^{k_0}(\zeta\cdot\nabla\bar\rho)\|_{L^2} \nonumber\\
		&\lesssim
		\|\zeta\|_{L^2}
		\left(
		\|\nabla^{k_0+1}\bar u\|_{L^\infty}
		+
		\|\nabla^{k_0+1}\bar\rho\|_{L^\infty}
		\right) \nonumber\\
		&\quad
		+
		\|\nabla^{k_0}\zeta\|_{L^2}
		\left(
		\|\nabla\bar u\|_{L^\infty}
		+
		\|\nabla\bar\rho\|_{L^\infty}
		\right) \nonumber\\
		&\lesssim
		\lambda^{k_0-1+\frac{2}{p}-\beta}
		\mathcal E_\lambda
		+
		\|\nabla^{k_0}\zeta\|_{L^2}
		\lambda^{1+\frac{2}{p}-\beta}
		\left(
	\left(\log\log\lambda\right)^{-\frac{|\beta|}{2+|\beta|}}+	\left(\log\log\lambda\right)^{\frac{2-\beta}{2+|\beta|}}
	\right).
		\label{product-k0-estimate-inviscid}
	\end{align}
	For the commutators, \eqref{bootstrap-assumption-inviscid} and
	\eqref{estimate-bar-u} imply
	\begin{align}\nonumber
		&\left\|
		\nabla^{k_0}\big((\bar u-\zeta)\cdot\nabla\zeta\big)
		-(\bar u-\zeta)\cdot\nabla\nabla^{k_0}\zeta
		\right\|_{L^2} \nonumber\\
		&\quad+
		\left\|
		\nabla^{k_0}\big((\bar u-\zeta)\cdot\nabla\xi\big)
		-(\bar u-\zeta)\cdot\nabla\nabla^{k_0}\xi
		\right\|_{L^2} \nonumber\\\nonumber
		&\lesssim
		\|\nabla^{k_0}\zeta,\nabla^{k_0}\xi\|_{L^2}
		\left(
		\lambda^{1+\frac{2}{p}-\beta}\left(\log\log\lambda\right)^{-\frac{|\beta|}{2+|\beta|}}
		+
		\exp(C_0L^2)
		\right)\\
&~~~~~~~~		+
		\lambda^{k_0+\frac{2}{p}-\beta}	\left(\log\log\lambda\right)^{-\frac{|\beta|}{2+|\beta|}}
		\|\nabla\zeta,\nabla\xi\|_{L^2}.
		\label{commutator-k0-estimate-inviscid}
	\end{align}
	Combining \eqref{G-k0-estimate-inviscid}, \eqref{product-k0-estimate-inviscid},
	and \eqref{commutator-k0-estimate-inviscid}, we obtain
	\begin{align}\nonumber
		&\frac{d}{dt}
		\|\nabla^{k_0}\zeta,\nabla^{k_0}\xi\|_{L^2}\\\nonumber
		&\lesssim
		\lambda^{k_0-1+\frac{2}{p}-\beta}
		\mathcal E_\lambda
		+
		\lambda^{k_0+\frac{2}{p}-\beta}	\left(\log\log\lambda\right)^{-\frac{|\beta|}{2+|\beta|}}
	\|\nabla\zeta,\nabla\xi\|_{L^2} \nonumber\\
		&\quad
		+
		\|\nabla^{k_0}\zeta,\nabla^{k_0}\xi\|_{L^2}
		\left(
		\lambda^{1+\frac{2}{p}-\beta}
		\left(
	\left(\log\log\lambda\right)^{-\frac{|\beta|}{2+|\beta|}}+	\left(\log\log\lambda\right)^{\frac{2-\beta}{2+|\beta|}}
	\right)
		+
		\exp(C_0L^2)
		\right).
		\label{k0-L2norm-inviscid}
	\end{align}
		Because $\beta-\frac{2}{p}<1$,
		\[
		t^*\exp(C_0L^2)
		=
		\lambda^{\beta-\frac{2}{p}-1+o(1)}L
		\longrightarrow0.
		\]
		Taking \(k_0=1\) and applying Gronwall's inequality therefore yields
	\begin{align}\label{nabla-zeta-xi-errorL2-inviscid}
		\|\nabla\zeta,\nabla\xi\|_{L^2}
		\lesssim
		\lambda^{-1}\mathcal E_\lambda.
	\end{align}
	For \(k_0>1\), substituting \eqref{nabla-zeta-xi-errorL2-inviscid} into
	\eqref{k0-L2norm-inviscid} and applying Gronwall once more gives
	\begin{align}\label{error-k0-L2-inviscid}
		\|\nabla^{k_0}\zeta,\nabla^{k_0}\xi\|_{L^2}
		\lesssim
		\lambda^{k_0-2}\mathcal E_\lambda.
	\end{align}
	This proves the \(L^2\) estimate for all \(k\in\mathbb N\).
	
	\medskip
	\noindent{\bf Step 3. A low-integrability estimate.}
	The case \(q>2\) will follow by interpolation from the \(L^2\) estimates. We
	therefore prove a separate estimate for \(1<q<2\). Let \(\eta\in(1,2)\). Taking the
	\(L^\eta\) estimate of \eqref{pertur-system-inviscid}, we obtain
	\begin{align}\nonumber
		\frac{d}{dt}\|\zeta,\xi\|_{L^\eta}
		&\lesssim
		\|\xi\|_{L^\eta}
		+\|G\|_{L^\eta}
		+\|\nabla(\bar p-p)\|_{L^\eta} \nonumber\\
		&\quad
		+\|\zeta\cdot\nabla\bar u\|_{L^\eta}
		+\|\zeta\cdot\nabla\bar\rho\|_{L^\eta}.
		\label{error-Leta-inviscid}
	\end{align}
	By the Calder\'on--Zygmund estimate and H\"older's inequality,
	\begin{align}\nonumber
		\|\nabla(\bar p-p)\|_{L^\eta}
		&\lesssim
		\|\bar u\cdot\nabla\zeta+\zeta\cdot\nabla\bar u-\zeta\cdot\nabla\zeta\|_{L^\eta}
		+\|\xi\|_{L^\eta}+\|G\|_{L^\eta}
		\nonumber\\
		&\leq
		\|\nabla\zeta\|_{L^2}
		\|\bar u\|_{L^{\frac{2\eta}{2-\eta}}}
		+
		\|\zeta\|_{L^\eta}
		\left(
		\|\nabla\bar u\|_{L^\infty}
		+
		\|\nabla\zeta\|_{L^\infty}
		\right)+\|\xi\|_{L^\eta}+\|G\|_{L^\eta}.
		\label{pressure-Leta-inviscid}
	\end{align}
	
	We also record the Lebesgue estimate for the approximate velocity. By the
	scaling of the Biot--Savart law,
	\[
	\bar u(t,x)
	=
	\lambda^{\frac{2}{p}-\beta}
	L^{-\frac{|\beta|}{2+|\beta|}}
	u[f](\lambda x),
	\qquad
	L=\log\log\lambda.
	\]
	Since \(u[f]\in C_c^\infty(\mathbb R^2)\), it follows that, for every
	\(1\leq r\leq\infty\),
	\begin{align}\label{bar-u-Lr-inviscid}
		\|\bar u(t)\|_{L^r}
		\lesssim
		\lambda^{\frac{2}{p}-\frac{2}{r}-\beta}
		L^{-\frac{|\beta|}{2+|\beta|}}.
	\end{align}
	In particular, taking
	\[
	r=\frac{2\eta}{2-\eta},
	\qquad
	\frac{2}{r}=\frac{2}{\eta}-1,
	\]
	and using \eqref{nabla-zeta-xi-errorL2-inviscid}, we obtain
	\[
	\|\nabla\zeta\|_{L^2}
	\|\bar u\|_{L^{\frac{2\eta}{2-\eta}}}
	\lesssim
	\lambda^{\frac{2}{p}-\frac{2}{\eta}-\beta}
	\mathcal E_\lambda.
	\]
	Using \eqref{bar-u-Lr-inviscid}, \eqref{estimate-bar-u},
	\eqref{estimate-G-inviscid}, \eqref{bootstrap-assumption-inviscid}, and
	\eqref{nabla-zeta-xi-errorL2-inviscid}, we infer
	\begin{align}\nonumber
		\frac{d}{dt}\|\zeta,\xi\|_{L^\eta}
		&\lesssim
		\lambda^{\frac{2}{p}-\frac{2}{\eta}-\beta}
		\mathcal E_\lambda \nonumber\\
		&\quad
		+
		\|\zeta,\xi\|_{L^\eta}
		\left(
	\lambda^{1+\frac{2}{p}-\beta}
	\left(
	\left(\log\log\lambda\right)^{-\frac{|\beta|}{2+|\beta|}}+	\left(\log\log\lambda\right)^{\frac{2-\beta}{2+|\beta|}}
	\right)
		+
		\exp(C_0L^2)
	\right).
		\label{Leta-differential-inviscid}
	\end{align}
	Gronwall's inequality gives
	\begin{align}\label{error-estimate-Leta-inviscid}
		\|\zeta,\xi\|_{L^\eta}
		\lesssim
		\lambda^{-1-\frac{2}{\eta}}
		\mathcal E_\lambda.
	\end{align}
	
	\medskip
	\noindent{\bf Step 4. General \(L^q\) estimates and closure of the bootstrap.}
		We interpolate separately below and above $L^2$. If $1<q<2$, choose
		$\eta=(q+1)/2$ and then an integer $m>k$ sufficiently large. Choose
		$0<\theta<1$ so that
		\[
		\frac1q-\frac{k}{2}
		=
		\frac{1-\theta}{\eta}
		+\theta\left(\frac12-\frac m2\right).
		\]
		The Gagliardo--Nirenberg inequality between $L^\eta$ and $\dot W^{m,2}$,
		together with \eqref{error-estimate-Leta-inviscid} and
		\eqref{error-k0-L2-inviscid}, gives the required estimate. If
		$2\leq q<\infty$, we instead choose $m$ and $\theta$ from
		\[
		\frac1q-\frac{k}{2}
		=
		\frac{1-\theta}{2}
		+\theta\left(\frac12-\frac m2\right)
		\]
		and interpolate between the $L^2$ and high-order $L^2$ bounds. In both cases,
		the interpolation relation gives
		\begin{align}\label{error-any-k-q-inviscid}
			\|\nabla^k\zeta,\nabla^k\xi\|_{L^q}
			\lesssim
			\lambda^{k-1-\frac{2}{q}}\mathcal E_\lambda.
		\end{align}
		For $q=\infty$, the same conclusion follows from the $q=4$ estimates and
		\[
		\|\nabla^k\zeta,\nabla^k\xi\|_{L^\infty}
		\lesssim
		\|\nabla^k\zeta,\nabla^k\xi\|_{L^4}^{1/2}
		\|\nabla^{k+1}\zeta,\nabla^{k+1}\xi\|_{L^4}^{1/2}.
		\]
		It remains only to verify \eqref{bootstrap-assumption-inviscid}. Taking \(q=4\) in
	\eqref{error-any-k-q-inviscid}, we get
	\[
	\|\nabla\zeta,\nabla\xi\|_{L^\infty}
	\lesssim
		\|\nabla^2\zeta,\nabla^2\xi\|_{L^4}^{\frac{1}{2}}
		\|\nabla\zeta,\nabla\xi\|_{L^4}^{\frac{1}{2}}
		\leq
		\mathcal E_\lambda.
		\]
		Choosing \(C_0\) sufficiently large, this improves the bootstrap bound.
		A standard continuity argument then shows that \eqref{bootstrap-assumption-inviscid}
		holds on \([0,t^*]\). The estimate \eqref{error-any-k-q-inviscid} is therefore valid
		on the whole interval, which proves \eqref{error-estimate-short-time-inviscid}.
		Running the argument on the maximal smooth existence interval, these estimates also
		control the continuation norms. Standard local existence and continuation for smooth
		inviscid Boussinesq solutions therefore extend the exact solution through $[0,t^*]$.
\end{proof}

As a consequence, we obtain the following Besov version of the error estimate. This
form will be used to compare the exact and approximate solutions in the norm inflation
space.

\begin{coro}\label{coro-error-besov-inviscid}
	For any $s>0$ and $1<q\leq\infty$, there holds
	\begin{align}\label{error-besov-inviscid}
		\|\zeta(t), \xi(t)\|_{\dot B^s_{q,1}}
		\lesssim
		\lambda^{s-1-\frac{2}{q}}
		\mathcal E_\lambda,
	\end{align}
	for $0<t\leq t^*$.
\end{coro}

\begin{proof}
	Let \(m=\lfloor s\rfloor+1\). By interpolation and the embeddings
	\(L^q\hookrightarrow \dot B^0_{q,\infty}\), 
	\(\dot W^{m,q}\hookrightarrow \dot B^m_{q,\infty}\), we have
	\begin{align*}
		\|\zeta(t), \xi(t)\|_{\dot B^s_{q,1}}
		&\lesssim
		\|\zeta(t), \xi(t)\|_{L^q}^{1-\frac{s}{m}}
		\|\nabla^m\xi(t),\nabla^m\zeta(t)\|_{L^q}^{\frac{s}{m}}.
	\end{align*}
	The desired estimate then follows from Lemma \ref{lem-error-estimate-inviscid} with
	\(k=0\) and \(k=m\).
\end{proof}

\subsection{Proof of Theorem \ref{thm}}\label{sec:3.3}

We now conclude the proof of norm inflation in the inviscid case. The initial
smallness follows from the estimates on the approximate velocity and density, while
the lower bound at the inflation time is obtained by combining the density growth with
the stability estimate.

\begin{proof}[Proof of Theorem \ref{thm}]
	Let \(\varepsilon>0\). We choose \(\lambda\gg1\) sufficiently large. Since the exact
	and approximate solutions have the same initial data,
	\[
	u_0=\bar u(0),
	\qquad
	\rho_0=\bar\rho(0).
	\]
	By Lemmas \ref{lem-bar-u-beta-estimate} and \ref{lem-rho-lower-upper-bound},
	\[
	\|u_0,\rho_0\|_{\dot B^\beta_{p,1}}
	\lesssim
	(\log\log\lambda)^{-\frac{|\beta|}{2+|\beta|}}.
	\]
	Hence, for \(\lambda\) sufficiently large,
	\[
	\|u_0,\rho_0\|_{\dot B^\beta_{p,1}}
<\varepsilon.
	\]
	
	We next prove the lower bound at the inflation time. Since
	\[
	-2<\beta-\frac{2}{p}<1,
	\]
		we may choose \(0<\delta<1\) so small that
		\[
		\sigma\define \beta-\frac{2}{p}+2-\delta>0,
		\qquad
		\frac{2}{2-\delta}\leq p.
	\]
	Set
	\[
	\tau\define \frac{2}{2-\delta}.
	\]
	Then
	\[
	\beta-\frac{2}{p}
	=
	\sigma-\frac{2}{\tau},
	\]
	and hence the Besov embedding gives
	\[
	\dot B^\sigma_{\tau,1}
	\hookrightarrow
	\dot B^\beta_{p,\infty}.
	\]
	By Corollary \ref{coro-error-besov-inviscid},
	\begin{align}\label{xi-error-final-inviscid}
		\|\zeta(t^*),\xi(t^*)\|_{\dot B^\beta_{p,\infty}}
		&\lesssim
		\|\zeta(t^*),\xi(t^*)\|_{\dot B^\sigma_{\tau,1}} \notag\\
		&\lesssim
		\lambda^{\sigma-1-\frac{2}{\tau}}
		\mathcal E_\lambda \notag\\
		&=
		\lambda^{\beta-\frac{2}{p}-1}
		\mathcal E_\lambda.
	\end{align}
	Since \(\beta-\frac{2}{p}<1\), the last quantity tends to zero faster than any
	negative power of \(\log\log\lambda\). Thus, for \(\lambda\) sufficiently large,
	\[
	\|\zeta(t^*),\xi(t^*)\|_{\dot B^\beta_{p,\infty}}
	\leq
	\frac{1}{10}
	(\log\log\lambda)^{-\frac{|\beta|}{2+|\beta|}}.
	\]
	
	Using Lemma \ref{lem-rho-lower-upper-bound}, we obtain
	\begin{align*}
		\|\rho(t^*)\|_{\dot B^\beta_{p,\infty}}
		&\geq
		\|\bar\rho(t^*)\|_{\dot B^\beta_{p,\infty}}
		-
		\|\xi(t^*)\|_{\dot B^\beta_{p,\infty}}\\
		&\gtrsim
		(\log\log\lambda)^{\frac{|\beta|}{2+|\beta|}}.
	\end{align*}
	After increasing \(\lambda\) if necessary, we get
	\[
	\|\rho(t^*)\|_{\dot B^\beta_{p,\infty}}
>
	\varepsilon^{-1}.
	\]
	
	We also record that the velocity component of the exact solution stays small in the
	same norm. Indeed, by Lemma \ref{lem-bar-u-beta-estimate} and
	\eqref{xi-error-final-inviscid},
	\[
	\|u(t^*)\|_{\dot B^\beta_{p,\infty}}
	\leq
	\|\bar u(t^*)\|_{\dot B^\beta_{p,1}}
	+
	\|\zeta(t^*)\|_{\dot B^\beta_{p,\infty}}
	\lesssim
	(\log\log\lambda)^{-\frac{|\beta|}{2+|\beta|}}.
	\]
	Taking \(\lambda\) larger if necessary, we have
	\[
	\|u(t^*)\|_{\dot B^\beta_{p,\infty}}
<\varepsilon.
	\]
	Thus the norm inflation is produced by the density component.
	
	Finally,
	\[
	t^*
	=
	\lambda^{-\frac{2}{p}-1+\beta}\log\log\lambda
	=
	\lambda^{\beta-\frac{2}{p}-1}\log\log\lambda
	\to0
	\qquad \text{as } \lambda\to\infty,
	\]
	because \(\beta-\frac{2}{p}<1\). Taking \(\lambda\) larger if necessary, we also have
	\(t^*<\varepsilon\). This proves norm inflation in arbitrarily short time and
	completes the proof.
\end{proof}
\section{The Fully Dissipative Case}\label{sec:4}

In this section, we prove the norm inflation result for the two-dimensional fully
dissipative Boussinesq system. We use the same profiles \(f,g\) and the same
approximate solution as in the inviscid case. The only new terms are the Laplacians,
which will be treated as additional errors.

\subsection{Approximate Solution and Stability of Approximation}\label{sec:4.1}

We first recall the approximate solution constructed in the inviscid case. The pair
\((\bar w,\bar\rho)\) satisfies
\[
\partial_t\bar w+\bar u\cdot\nabla\bar w=0,
\qquad
\partial_t\bar\rho+\bar u\cdot\nabla\bar\rho=0,
\qquad
\bar u=u[\bar w].
\]
In the velocity formulation, this can be written as
\[
\begin{cases}
	\partial_t\bar u+\bar u\cdot\nabla\bar u+\nabla\bar p
	=
	\bar\rho e_2+G,\\[1mm]
	\partial_t\bar\rho+\bar u\cdot\nabla\bar\rho=0,
\end{cases}
\]
where
\[
G=-\bar\rho e_2.
\]
When the same approximate solution is inserted into the fully dissipative system, the
Laplacians produce two additional error terms:
\[
\begin{cases}
	\partial_t \bar u-\Delta\bar u+\bar u\cdot\nabla\bar u+\nabla\bar p
	=
	\bar\rho e_2+G+G_{\mathrm{vis}},\\[1mm]
	\partial_t \bar\rho-\Delta\bar\rho+\bar u\cdot\nabla\bar\rho
	=
	R_{\mathrm{vis}},
\end{cases}
\]
with
\[
G_{\mathrm{vis}}=-\Delta\bar u,
\qquad
R_{\mathrm{vis}}=-\Delta\bar\rho.
\]

The estimates on the approximate solution obtained in the inviscid case remain
available here. By Lemma \ref{app-w-bound-s geq 0}, Corollary
\ref{app-u-bound-s geq 0}, and Lemma \ref{app-rho-bound-s geq 0}, for every
\(s>0\) and \(1\leq q\leq\infty\),
\begin{align}\label{estimate-bar-u-fully}
	\|\bar u(t)\|_{\dot B^s_{q,1}}
	\lesssim
	\lambda^{s+\frac{2}{p}-\frac{2}{q}-\beta} \left(\log\log \lambda\right)^{-\frac{|\beta|}{2+|\beta|}},
\end{align}
and
\begin{align}\label{estimate-bar-rho-fully}
	\|\bar\rho(t)\|_{\dot B^{s}_{q,1}}
	\lesssim
	\lambda^{s+\frac{2}{p}-\frac{2}{q}-\beta}
	(\log\log\lambda)^{\frac{2s-\beta}{2+|\beta|}},
\end{align}
for all \(0\leq t\leq t^*\). Moreover, Lemmas
\ref{lem-bar-u-beta-estimate} and \ref{lem-rho-lower-upper-bound} give
\begin{align}\label{bar-u-beta-estimate-fully}
	\|\bar u(t)\|_{\dot B^\beta_{p,1}}
	=
	\|\bar u(0)\|_{\dot B^\beta_{p,1}}
	\lesssim \left(\log\log \lambda\right)^{-\frac{|\beta|}{2+|\beta|}},
	\qquad 0\leq t\leq t^*,
\end{align}
and
\begin{align}\label{rho-initial-small-besov-fully}
	\|\rho(0)\|_{\dot B^\beta_{p,1}}
	=
	\|\bar\rho(0)\|_{\dot B^\beta_{p,1}}
	\lesssim
\left(\log\log \lambda\right)^{-\frac{|\beta|}{2+|\beta|}},
\end{align}
while at \(t=t^*\),
\begin{align}\label{rho-final-large-besov-fully}
	\|\bar\rho(t^*)\|_{\dot B^\beta_{p,\infty}}
	\gtrsim
	(\log\log\lambda)^{\frac{|\beta|}{2+|\beta|}}.
\end{align}

We now compare the approximate solution with the exact solution of the fully
dissipative Boussinesq system
\[
\begin{cases}
	\partial_t u-\Delta u+u\cdot\nabla u+\nabla p=\rho e_2,\\[1mm]
	\partial_t \rho-\Delta\rho+u\cdot\nabla\rho=0.
\end{cases}
\]
Let \((u,\rho)\) be the exact solution with the same initial data as
\((\bar u,\bar\rho)\), and set
\[
\zeta\define \bar u-u,
\qquad
\xi\define \bar\rho-\rho .
\]
Subtracting the exact system from the approximate one gives
\begin{align}\label{pertur-system-fully}
	\begin{cases}
		\partial_t\zeta-\Delta\zeta+\bar u\cdot\nabla\zeta-\zeta\cdot\nabla\zeta
		+\nabla(\bar p-p)
		=
		-\zeta\cdot\nabla\bar u+\xi e_2+G+G_{\mathrm{vis}},\\[1mm]
		\partial_t\xi-\Delta\xi+\bar u\cdot\nabla\xi-\zeta\cdot\nabla\xi
		=
		-\zeta\cdot\nabla\bar\rho+R_{\mathrm{vis}},
	\end{cases}
\end{align}
with
\[
\zeta(0)=0,
\qquad
\xi(0)=0.
\]
The pressure difference satisfies
\begin{align}\label{pressure-fully}
\bar p-p
=
(-\Delta)^{-1}\diver\left(
\bar u\cdot\nabla\zeta
+\zeta\cdot\nabla\bar u
-\zeta\cdot\nabla\zeta
-\xi e_2-G-G_{\mathrm{vis}}
\right).
\end{align}
Notice that $\diver G_{\mathrm{vis}}=0$, because $\bar u$ is divergence free.
Compared with the inviscid perturbation system, the only new source terms are
\(G_{\mathrm{vis}}\) and \(R_{\mathrm{vis}}\). 

We record the bounds for the forcing terms. From \eqref{estimate-bar-rho-fully},
\begin{align}\label{estimate-G-fully}
	\|G(t)\|_{\dot W^{s,q}}
	\leq
	\|\bar\rho(t)\|_{\dot W^{s,q}}
	\lesssim
	\lambda^{s+\frac{2}{p}-\frac{2}{q}-\beta}
	(\log\log\lambda)^{\frac{2s-\beta}{2+|\beta|}}.
\end{align}
Furthermore, by \eqref{estimate-bar-u-fully} and \eqref{estimate-bar-rho-fully},
\begin{align}\label{estimate-Gvis-fully}
	\|G_{\mathrm{vis}}(t)\|_{\dot W^{s,q}}
	\lesssim
	\|\bar u(t)\|_{\dot B^{s+2}_{q,1}}
	\lesssim
	\lambda^{s+2+\frac{2}{p}-\frac{2}{q}-\beta}	(\log\log\lambda)^{-\frac{|\beta|}{2+|\beta|}},
\end{align}
and
\begin{align}\label{estimate-Rvis-fully}
	\|R_{\mathrm{vis}}(t)\|_{\dot W^{s,q}}
	\lesssim
	\|\bar\rho(t)\|_{\dot W^{s+2,q}}
	\lesssim
	\lambda^{s+2+\frac{2}{p}-\frac{2}{q}-\beta}
	(\log\log\lambda)^{\frac{2s+4-\beta}{2+|\beta|}}.
\end{align}
Consequently,
\begin{align}\label{estimate-forcing-fully}
	\|G(t),G_{\mathrm{vis}}(t),R_{\mathrm{vis}}(t)\|_{\dot W^{s,q}}
	\lesssim
	\lambda^{s+2+\frac{2}{p}-\frac{2}{q}-\beta}
	(\log\log\lambda)^{\frac{2s+4-\beta}{2+|\beta|}},
\end{align}
for \(0\leq t\leq t^*\).

The viscous forcing terms contain two additional derivatives. The following proof
records how this changes the error bounds and where the stronger restriction
$\beta-\frac{2}{p}<-1$ is used.

\begin{coro}\label{coro-error-besov-fully}
	For any \(s>0\) and \(1<q\leq\infty\), one has
	\begin{align}\label{error-besov-fully}
		\|\xi(t),\zeta(t)\|_{\dot B^s_{q,1}}
		\lesssim
		\lambda^{s+1-\frac{2}{q}}
		\mathcal E_\lambda,
	\end{align}
	for all \(0<t\leq t^*\).
\end{coro}

\begin{proof}
	Write $E=(\zeta,\xi)$. We run the estimates on a bootstrap interval on which
	\begin{align}\label{bootstrap-assumption-fully}
		\|\nabla E\|_{L^\infty}
		\leq
		\lambda^2\exp(C_0L^2).
	\end{align}
	The diffusion terms on the left-hand side of \eqref{pertur-system-fully} are
	nonnegative in the $L^2$ energy identities and may be discarded for upper bounds.
	Using \eqref{estimate-forcing-fully} with $s=0$ and $q=2$, the same product
	estimates as in the proof of Lemma \ref{lem-error-estimate-inviscid}, and
	\eqref{exp-gronwall-bound-1-inviscid}, we obtain
	\begin{align}\label{error-L2-fully}
		\|E(t)\|_{L^2}
		\lesssim
	\mathcal E_\lambda,
		\qquad 0\leq t\leq t^*.
	\end{align}
	Indeed, the largest source has size
	$\lambda^{1+\frac{2}{p}-\beta}$ up to powers of $L$, and its time integral is
	bounded by a power of $L$. For an integer $k\geq1$, differentiating the system
	$k$ times and using the standard commutator estimate gives inductively
	\begin{align}\label{error-high-L2-fully}
		\|\nabla^kE(t)\|_{L^2}
		\lesssim
		\lambda^k\mathcal E_\lambda.
	\end{align}
	Here the differentiated forcing is bounded by
	$\lambda^{k+1+\frac{2}{p}-\beta}$ up to powers of $L$; multiplication by $t^*$
	leaves $\lambda^k$ times a power of $L$. The terms involving the bootstrap norm do
	not enlarge the Gronwall factor, because
	\begin{align}\label{bootstrap-time-factor-fully}
		t^*\lambda^2\exp(C_0L^2)
		=
		\lambda^{\beta-\frac{2}{p}+1+o(1)}L
		\longrightarrow0
	\end{align}
	under the assumption $\beta-\frac{2}{p}<-1$.

For \(1<\eta<2\), we next derive the low-integrability estimate.
The Laplacian is dissipative in \(L^\eta\), and hence the diffusion terms
may again be discarded for an upper bound. Using the pressure formula \eqref{pressure-fully}, the
Calder\'on--Zygmund theorem, and H\"older's inequality as in Step~3 of
Lemma \ref{lem-error-estimate-inviscid}, we obtain
\begin{align*}
	\frac{d}{dt}\|E(t)\|_{L^\eta}
	&\lesssim
	\|G,G_{\mathrm{vis}},R_{\mathrm{vis}}\|_{L^\eta}
	+
	\|\nabla\zeta\|_{L^2}
	\|\bar u\|_{L^{\frac{2\eta}{2-\eta}}}
	\\
	&\quad+
	\|E(t)\|_{L^\eta}
	\left(
	1+\|\nabla\bar u\|_{L^\infty}
	+\|\nabla\bar\rho\|_{L^\infty}
	+\|\nabla E\|_{L^\infty}
	\right).
\end{align*}
Here \(G_{\mathrm{vis}}\) does not contribute to the pressure estimate because
\(\diver G_{\mathrm{vis}}=0\).

By \eqref{estimate-forcing-fully}, \eqref{error-high-L2-fully},
\eqref{bar-u-Lr-inviscid}, and the bootstrap assumption
\eqref{bootstrap-assumption-fully}, it follows that
\begin{align*}
	\frac{d}{dt}\|E(t)\|_{L^\eta}
	&\lesssim
	\lambda^{2+\frac{2}{p}-\frac{2}{\eta}-\beta}
	\mathcal E_\lambda
	\\
	&\quad+
	\|E(t)\|_{L^\eta}
	\Bigg[
	1+
	\lambda^{1+\frac{2}{p}-\beta}
	\left(
	L^{-\frac{|\beta|}{2+|\beta|}}
	+
	L^{\frac{2-\beta}{2+|\beta|}}
	\right)
	+
	\lambda^2\exp(C_0L^2)
	\Bigg].
\end{align*}
Indeed, since
\[
\|\nabla\zeta\|_{L^2}\lesssim\lambda\mathcal E_\lambda
\]
and
\[
\|\bar u\|_{L^{\frac{2\eta}{2-\eta}}}
\lesssim
\lambda^{1+\frac{2}{p}-\frac{2}{\eta}-\beta}
L^{-\frac{|\beta|}{2+|\beta|}},
\]
their product is bounded by the first term on the right-hand side.

Moreover, \eqref{exp-gronwall-bound-1-inviscid} and
\eqref{bootstrap-time-factor-fully} imply that
\[
t^*
\left[
1+
\lambda^{1+\frac{2}{p}-\beta}
\left(
L^{-\frac{|\beta|}{2+|\beta|}}
+
L^{\frac{2-\beta}{2+|\beta|}}
\right)
+
\lambda^2\exp(C_0L^2)
\right]
\lesssim L^2+o(1).
\]
Since
\[
t^*
\lambda^{2+\frac{2}{p}-\frac{2}{\eta}-\beta}
=
\lambda^{1-\frac{2}{\eta}}L,
\]
Gronwall's inequality gives
\begin{align}\label{error-low-fully}
	\|E(t)\|_{L^\eta}
	\lesssim
	\lambda^{1-\frac{2}{\eta}}
	\mathcal E_\lambda,
	\qquad
	0\leq t\leq t^*.
\end{align}
	Interpolating \eqref{error-low-fully}, \eqref{error-L2-fully}, and
	\eqref{error-high-L2-fully} exactly as in Step 4 of Lemma
	\ref{lem-error-estimate-inviscid}, we find, for $k\in\mathbb N_0$ and
	$1<q\leq\infty$,
	\begin{align}\label{error-any-k-q-fully}
		\|\nabla^kE(t)\|_{L^q}
		\lesssim
		\lambda^{k+1-\frac{2}{q}}\mathcal E_\lambda.
	\end{align}
	In particular, the $q=4$ estimates and the Gagliardo--Nirenberg inequality give
	\[
	\|\nabla E\|_{L^\infty}
	\lesssim
	\|\nabla E\|_{L^4}^{1/2}
	\|\nabla^2E\|_{L^4}^{1/2}
	\lesssim
	\lambda^2\mathcal E_\lambda.
	\]
	Choosing $C_0$ sufficiently large closes \eqref{bootstrap-assumption-fully} by
	continuity. Finally, interpolation between the estimates in
	\eqref{error-any-k-q-fully} gives
	\[
	\|E(t)\|_{\dot B^s_{q,1}}
	\lesssim
	\lambda^{s+1-\frac{2}{q}}\mathcal E_\lambda,
	\]
	which is \eqref{error-besov-fully}.
\end{proof}

\subsection{Proof of Theorem \ref{thm-1}}\label{sec:4.2}

We now conclude the proof of norm inflation in the fully dissipative case. The argument
is parallel to the inviscid case, except that the stability estimate loses two additional
derivatives. This loss is consistent with the parabolic scaling and is absorbed by the
condition \(\beta-\frac{2}{p}<-1\).

\begin{proof}[Proof of Theorem \ref{thm-1}]
	Let \(\varepsilon>0\). We choose \(\lambda\gg1\) sufficiently large. Since the exact
	and approximate solutions have the same initial data, \eqref{bar-u-beta-estimate-fully}
	and \eqref{rho-initial-small-besov-fully} give
	\[
	\|u_0,\rho_0\|_{\dot B^\beta_{p,1}}
	\lesssim
	(\log\log \lambda)^{-\frac{|\beta|}{2+|\beta|}}.
	\]
	Hence, for \(\lambda\) sufficiently large,
	\[
	\|u_0,\rho_0\|_{\dot B^\beta_{p,1}}
<\varepsilon.
	\]
	
	We next prove the lower bound at the inflation time. Since
	\[
	-2<\beta-\frac{2}{p}<-1,
	\]
		we may choose \(0<\delta<1\) so small that
		\[
		\sigma\define \beta-\frac{2}{p}+2-\delta>0,
		\qquad
		\frac{2}{2-\delta}\leq p.
	\]
	Set
	\[
	\tau\define \frac{2}{2-\delta}.
	\]
	Then
	\[
	\beta-\frac{2}{p}
	=
	\sigma-\frac{2}{\tau},
	\]
	and the Besov embedding gives
	\[
	\dot B^\sigma_{\tau,1}
	\hookrightarrow
	\dot B^\beta_{p,\infty}.
	\]
	By Corollary \ref{coro-error-besov-fully},
	\begin{align}\label{xi-error-final-fully}
		\|\zeta(t^*),\xi(t^*)\|_{\dot B^\beta_{p,\infty}}
		&\lesssim
		\|\zeta(t^*),\xi(t^*)\|_{\dot B^\sigma_{\tau,1}} \notag\\
		&\lesssim
		\lambda^{\sigma+1-\frac{2}{\tau}}
		\mathcal E_\lambda \notag\\
		&=
		\lambda^{\beta-\frac{2}{p}+1}
		\mathcal E_\lambda.
	\end{align}
	Since \(\beta-\frac{2}{p}<-1\), the last quantity tends to zero faster than any
	negative power of \(\log\log\lambda\). Thus, for \(\lambda\) sufficiently large,
	\[
	\|\zeta(t^*),\xi(t^*)\|_{\dot B^\beta_{p,\infty}}
	\leq
	\frac{1}{10}
	(\log\log\lambda)^{-\frac{|\beta|}{2+|\beta|}}.
	\]
	
	Using \eqref{rho-final-large-besov-fully}, we obtain
	\begin{align*}
		\|\rho(t^*)\|_{\dot B^\beta_{p,\infty}}
		&\geq
		\|\bar\rho(t^*)\|_{\dot B^\beta_{p,\infty}}
		-
		\|\xi(t^*)\|_{\dot B^\beta_{p,\infty}}\\
		&\gtrsim
		(\log\log\lambda)^{\frac{|\beta|}{2+|\beta|}}.
	\end{align*}
	After increasing \(\lambda\) if necessary, we get
	\[
	\|\rho(t^*)\|_{\dot B^\beta_{p,\infty}}
>
	\varepsilon^{-1}.
	\]
	
	We also record that the velocity component of the exact solution stays small in the
	same norm. Indeed, by \eqref{bar-u-beta-estimate-fully} and
	\eqref{xi-error-final-fully},
	\[
	\|u(t^*)\|_{\dot B^\beta_{p,\infty}}
	\leq
	\|\bar u(t^*)\|_{\dot B^\beta_{p,1}}
	+
	\|\zeta(t^*)\|_{\dot B^\beta_{p,\infty}}
	\lesssim
	(\log\log\lambda)^{-\frac{|\beta|}{2+|\beta|}}.
	\]
	Taking \(\lambda\) larger if necessary, we have
	\[
	\|u(t^*)\|_{\dot B^\beta_{p,\infty}}
< \varepsilon.
	\]
	Thus the norm inflation is produced by the density component.
	
	Finally,
	\[
	t^*
	=
	\lambda^{\beta-\frac{2}{p}-1}\log\log\lambda
	\to0
	\qquad \text{as } \lambda\to\infty,
	\]
		because \(\beta-\frac{2}{p}<-1\). Taking \(\lambda\) larger if necessary, we also have
	\(t^*<\varepsilon\). This proves norm inflation in arbitrarily short time and
	completes the proof.
\end{proof}


	\appendix

\section{Choice of the Radial Profiles}\label{app-choice-fg}
In this appendix, we prove Lemma \ref{lem-choice-fg-inviscid}, which constructs the
radial profiles used in the approximate solution.
\begin{proof}[Proof of Lemma \ref{lem-choice-fg-inviscid}]
	Choose nonnegative nonzero functions
	\[
	\psi_0\in C_c^\infty((1/2,1)),\qquad
	\psi_1\in C_c^\infty((4,5)),\qquad
	\psi_2\in C_c^\infty((6,7)).
	\]
	We set
	\[
	f(r)=\psi_0(r)+a_1\psi_1(r)+a_2\psi_2(r),
	\]
	where \(a_1,a_2\) are chosen so that
	\begin{align}\label{moment-condition-f}
		\int_0^\infty f(r)\,dr=0,
		\qquad
		\int_0^\infty r f(r)\,dr=0.
	\end{align}
	Let
	\[
	A_j=\int_0^\infty \psi_j(r)\,dr,
	\qquad
	B_j=\int_0^\infty r\psi_j(r)\,dr,
	\qquad j=1,2.
	\]
	Since \(\psi_1\) and \(\psi_2\) are supported in disjoint annuli separated from each
	other, the ratios \(\frac{B_1}{A_1}\) and \(\frac{B_2}{A_2}\) are different. Hence
	\[
	\det
	\begin{pmatrix}
		A_1 & A_2\\
		B_1 & B_2
	\end{pmatrix}
	\neq0,
	\]
	and the two linear conditions in \eqref{moment-condition-f} uniquely determine
	\(a_1,a_2\).
	
	We next check that \(f^{(-2)}\) is compactly supported. With the convention
	\[
	f^{(-1)}(r)=\int_{-\infty}^r f(s)\,ds,
	\qquad
	f^{(-2)}(r)=\int_{-\infty}^r f^{(-1)}(s)\,ds,
	\]
	the support condition on \(f\) implies that \(f^{(-1)}=f^{(-2)}=0\) for \(r\leq0\).
		On the other hand, for \(r\) sufficiently large, the first condition in
		\eqref{moment-condition-f} gives
	\[
	f^{(-1)}(r)=\int_0^\infty f(s)\,ds=0.
	\]
	Moreover, by Fubini's theorem, for such \(r\),
	\[
	f^{(-2)}(r)
	=
	\int_{-\infty}^r f^{(-1)}(s)\,ds
	=
	\int_0^r\int_0^s f(\tau)\,d\tau\,ds
	=
	\int_0^r (r-\tau)f(\tau)\,d\tau.
	\]
	Since \(r\) is larger than the support of \(f\), this becomes
	\[
	f^{(-2)}(r)
	=
	r\int_0^\infty f(\tau)\,d\tau
	-
	\int_0^\infty \tau f(\tau)\,d\tau
	=0
	\]
	by \eqref{moment-condition-f}. Thus \(f^{(-2)}\) vanishes both near \(0\) and for
	large \(r\). Since \(f\in C_c^\infty((0,\infty))\), we conclude that
	\[
	f^{(-2)}\in C_c^\infty((0,\infty)).
	\]
	
	Now choose
	\[
	g\in C_c^\infty((2,3)),
	\qquad
	g\not\equiv0.
	\]
	Then
	\[
	\operatorname{supp}f\cap\operatorname{supp}g=\emptyset .
	\]
	
It remains to verify the nondegeneracy of the angular shear on
\(\operatorname{supp}g\). Since \(f\) is radial, the velocity \(u[f]\) is purely
angular and \(u_\theta[f]\) depends only on \(r\). We compute it at the point with
polar coordinates \((r,0)\), namely \(x=(r,0)\), where
\(\mathbf e_\theta=(0,1)\). Writing \(y=\rho(\cos\alpha,\sin\alpha)\), the
Biot--Savart law gives
	\[
	u_\theta[f](r)
	=
	c\int_0^\infty\int_{-\pi}^{\pi}
	\frac{r-\rho\cos\alpha}
	{r^2+\rho^2-2r\rho\cos\alpha}
	f(\rho)\rho\,d\alpha\,d\rho ,
	\]
	where \(c\neq0\) is a dimensional constant.
	
	We now compute the angular integral. Since
	\[
	\frac{r-\rho\cos\alpha}
	{r^2+\rho^2-2r\rho\cos\alpha}
	=
	\frac12\partial_r
	\log\left(r^2+\rho^2-2r\rho\cos\alpha\right),
	\]
	it is enough to evaluate
	\[
	\int_{-\pi}^{\pi}
	\log\left(r^2+\rho^2-2r\rho\cos\alpha\right)\,d\alpha .
	\]
	Using
	\[
	r^2+\rho^2-2r\rho\cos\alpha
	=
	|r-\rho e^{i\alpha}|^2,
	\]
	we have
	\[
	\log\left(r^2+\rho^2-2r\rho\cos\alpha\right)
	=
	2\log|r-\rho e^{i\alpha}|.
	\]
	The classical mean value identity for the logarithmic potential gives
	\[
	\int_{-\pi}^{\pi}\log|r-\rho e^{i\alpha}|\,d\alpha
	=
	2\pi\log\max\{r,\rho\}.
	\]
	Hence
	\[
	\int_{-\pi}^{\pi}
	\log\left(r^2+\rho^2-2r\rho\cos\alpha\right)\,d\alpha
	=
	4\pi\log\max\{r,\rho\}.
	\]
	Differentiating in \(r\), we obtain
	\[
	\int_{-\pi}^{\pi}
	\frac{r-\rho\cos\alpha}
	{r^2+\rho^2-2r\rho\cos\alpha}
	\,d\alpha
	=
	\frac12\partial_r\left(4\pi\log\max\{r,\rho\}\right).
	\]
	If \(0<\rho<r\), then \(\max\{r,\rho\}=r\), and the last expression equals
	\[
	\frac12\partial_r(4\pi\log r)=\frac{2\pi}{r}.
	\]
	If \(\rho>r\), then \(\max\{r,\rho\}=\rho\), which is independent of \(r\), and the
	expression equals \(0\). Therefore
	\[
	\int_{-\pi}^{\pi}
	\frac{r-\rho\cos\alpha}
	{r^2+\rho^2-2r\rho\cos\alpha}
	\,d\alpha
	=
	\begin{cases}
		\dfrac{2\pi}{r}, & 0<\rho<r,\\[2mm]
		0, & \rho>r.
	\end{cases}
	\]
	Substituting this identity into the Biot--Savart formula yields
	\[
	u_\theta[f](r)
	=
	\frac{C}{r}\int_0^r \rho f(\rho)\,d\rho
	\]
	for some constant \(C\neq0\).
	
	We now restrict to \(2<r<3\). In this range, the support of \(\psi_0\) lies inside
	\((0,r)\), whereas the supports of \(\psi_1\) and \(\psi_2\) lie outside \(r\).
	Consequently,
	\[
	\int_0^r \rho f(\rho)\,d\rho
	=
	\int_0^\infty \rho\psi_0(\rho)\,d\rho
	\define M_0.
	\]
	Since \(\psi_0\geq0\) and \(\psi_0\not\equiv0\), we have \(M_0>0\). Thus, on
	\((2,3)\),
	\[
	h(r)=\frac{u_\theta[f](r)}{r}
	=
	\frac{CM_0}{r^2},
	\qquad
	h'(r)=-\frac{2CM_0}{r^3}.
	\]
	Since \(C\neq0\), \(M_0>0\), and \(\operatorname{supp}g\subset(2,3)\), there exists
	\(c_0>0\) such that
	\[
	|h'(r)|\geq c_0
	\qquad \text{on } \operatorname{supp}g .
	\]
	This completes the proof.
\end{proof}

				\section*{Acknowledgments}
			The authors thank Professor Changxing Miao for helpful discussions related to this work. Q. Chen was partially supported by National Natural Science Foundation of China (Grant No: 12471149).


\begin{thebibliography}{00}
\bibitem{b-BCD-b-bahouri-chemin-fourier-book-2011}
H. Bahouri, J. Chemin and R. Danchin, {\em Fourier analysis and nonlinear partial differential equations}, Grundlehren der mathematischen Wissenschaften, 343, Springer, Heidelberg, 2011.

\bibitem{b-BHL-strong-2d-infty-boussinesq-3dEuler}
R. Bianchini, L. E. Hientzsch and F. Iandoli, {\em Strong ill-posedness in \(L^\infty\) of the 2D Boussinesq equations in vorticity form and application to the 3D axisymmetric Euler equations}, SIAM J. Math. Anal. 56 (2024), 5915--5968.


\bibitem{b-BL-euler-2-illpose-bourgain-2015-IA}
J. Bourgain and D. Li, {\em  Strong ill-posedness of the incompressible Euler equation in borderline Sobolev spaces}, Invent. Math. 201 (2015), 97--157.				
				
\bibitem{b-BL-euler-4-illpose-bourgain-2021-IMRN}
J. Bourgain and D. Li, {\em  Strong ill-posedness of the 3D incompressible Euler equation in borderline spaces},  Int. Math. Res. Not. (2021), 12155--12264. 				

\bibitem{b-BP08-NS-B-1infty-illposed}
J. Bourgain and N. Pavlovi{\'c}, {\em Ill-posedness of the {Navier}-{Stokes} equations in a critical space in {3D}}, J. Funct. Anal. 255 (2008), 2233--2247.


\bibitem{c-chae-Euler-local-besov}
D. Chae, {\em Local existence and blowup criterion for the {Euler} equations in the {Besov} spaces}, Asymptotic Anal. 38 (2004), 339--358.

\bibitem{c-Global regularity-partial viscosity terms}
D. Chae, {\em Global regularity for the 2D Boussinesq equations with partial viscosity terms}, Adv. Math. 203 (2006), 497--513.

\bibitem{c-CN-local-inviscid}
D. Chae and H. Nam, {\em Local existence and blow-up criterion for the {Boussinesq} equations}, Proc. R. Soc. Edinb., Sect. A, Math. 127 (1997), 935--946.

\bibitem{c-CH-2021-boussinesq-half-plane}
J. Chen and T. Hou, {\em Finite time blowup of 2D Boussinesq and 3D Euler equations with \(C^{1, \alpha}\) velocity and boundary}, Commun. Math. Phys. 383 (2021), 1559--1667. 
		
\bibitem{c-CH-2022-boussinesq-arxiv-self-similar}
J. Chen and T. Hou, {\em Stable nearly self-similar blowup of the {2D} {Boussinesq} and {3D} {Euler} equations with smooth data {I}: {Analysis}}, Preprint, arXiv:2210.07191, 2022.

\bibitem{c-CMZ-SQG-Ck-Sobolev-nonexistence-22}
D. C\'{o}rdoba and L. Martínez-Zoroa, {\em Non existence and strong ill-posedness in {{\(C^k\)}} and {Sobolev} spaces for {SQG}}, Adv. Math. 407 (2022), Article ID 108570, 74 p.

\bibitem{c-CMZ-SQG-cmp-generalized-24}
D. C\'{o}rdoba and L. Martínez-Zoroa, {\em Non-existence and strong ill-posedness in {{\(C^{k, \beta}\)}} for the generalized surface quasi-geostrophic equation}, Commun. Math. Phys. 405 (2024), Paper No. 170, 53 p.

\bibitem{c-CMZ-SQG-annpde-fractional diffusion}
D. C\'{o}rdoba and L. Martínez-Zoroa, {\em Global unique solutions with instantaneous loss of regularity for {SQG} with fractional diffusion}, Ann. PDE 10 (2024), No. 2, Paper No. 21, 52 p.

\bibitem{c-CMO-euler-2d-Instantaneous gap loss of Sobolev regularity-24}
D. C\'{o}rdoba, L. Martínez-Zoroa and W. S. O{\.z}a{\'n}ski, {\em Instantaneous gap loss of Sobolev regularity for the 2D incompressible Euler equations}, Duke Math. J. 173 (2024), 1931--1971.

\bibitem{c-CLM-boussinesq-finitetime-singular-2d-Boussinesq}
D. C{\'o}rdoba, A. La{\'{\i}}n-Sanclemente and L. Mart{\'{\i}}nez-Zoroa, {\em Finite-time singularity via multi-layer degenerate pendula for the 2D Boussinesq equation with uniform \(C^{1, \sqrt{\frac{4}{3}} - 1 - \varepsilon} \cap L^2\) force}, Adv. Math. 480 (2025), Article ID 110480, 205 p.

\bibitem{e-EJ-euler-simper}
T. Elgindi and I. Jeong, {\em Ill-posedness for the incompressible {Euler} equations in critical {Sobolev} spaces}, Ann. PDE 3 (2017), Paper No. 7, 19 p. 

\bibitem{e-EJ-boussinesq-finitetime-jiao}
T. Elgindi and I. Jeong, {\em Finite-time singularity formation for strong solutions to the Boussinesq system}, Ann. PDE 6 (2020), Paper No. 5, 50 p. 

\bibitem{e-EM-elgindi2020infty}
T. Elgindi and N. Masmoudi, {\em  $L^\infty$ ill-posedness for a class of equations arising in hydrodynamics},  Arch. Ration. Mech. Anal. 235 (2020), 1979--2025.

\bibitem{g-GL-IMRN-SQG-critical}
D. Guo and X. Luo, {\em Norm inflation for the critical SQG equation}, Int. Math. Res. Not. IMRN (2026), no. 11, rnag108.

\bibitem{j-JMWZ-boussinesq-critical dissipation-14siam}
Q. Jiu, C. Miao, J. Wu and Z. Zhang, {\em The 2D incompressible Boussinesq equations with general critical dissipation}, SIAM J. Math. Anal. 46 (2014), 3426--3454.

\bibitem{k-KP-commutator-Euler}
T. Kato and G. Ponce, {\em Commutator estimates and the Euler and Navier-Stokes equations}, Comm. Pure Appl. Math. 41 (1988), 891--907.

\bibitem{k-KJ-Euler-simple illposed}
J. Kim and I. Jeong, {\em A simple ill-posedness proof for incompressible {Euler} equations in critical {Sobolev} spaces}, J. Funct. Anal. 283 (2022), Article ID 109673, 34 p.

\bibitem{l-LW-norm-boussinesq-B-1-inftyinfty}
Z. Li and W. Wang, {\em Norm inflation for the Boussinesq system}, Discrete Contin. Dyn. Syst., Ser. B 26 (2021), 5449--5463.

\bibitem{l-L-euler-10-illpose-lxyt-2024-arxiv}
X. Luo, {\em Ill-posedness of incompressible fluids in supercritical Sobolev spaces}, arXiv preprint, arXiv:2404.07813, 2024.

\bibitem{l-L-euler-101-illpose-lxyt-2025-arxiv}
X. Luo, {\em Sharp norm inflation for 3D Navier-Stokes equations in supercritical spaces}, arXiv preprint, arXiv:2504.08288, 2025.

\bibitem{m-phys-boussinesq-Atmosphere and Ocean}
A. J. Majda, {\em Introduction to PDEs and Waves for the Atmosphere and Ocean}, Courant Lect. Notes Math. 9, AMS, New York, 2003.

\bibitem{m-MB01-Vorticity and Incompressible Flow}
A. J. Majda and A. L. Bertozzi, {\em Vorticity and Incompressible Flow}, Cambridge University Press, Cambridge, 2002.

\bibitem{m-MY-euler-illposed}
G. Misiołek and T. Yoneda, {\em Local ill-posedness of the incompressible Euler equations in \(C^1\) and \(B^1_{\infty,1}\)}, Math. Ann. 364 (2016), 243--268.


\bibitem{s-SW-critical dissipation-2019-JAM}
A. Stefanov and J. Wu, {\em A global regularity result for the 2D Boussinesq equations with critical dissipation}, J. Anal. Math. 137 (2019), 269--290.

\bibitem{s-SWXY-fractional-bouss-2025-math ann}
A. Stefanov, J. Wu, X. Xu and Z. Ye, {\em Global regularity results of the 2D fractional Boussinesq equations}, Math. Ann. 391 (2025), 5965--6012.

\bibitem{v-phys-boussinesq-2017-large-scale}
G. K. Vallis, {\em Atmospheric and oceanic fluid dynamics. Fundamentals and large-scale circulation}, 2nd edition. Cambridge: Cambridge University Press, 2017.

\bibitem{w-Wang15-ill-NS-B-1}
B. Wang, {\em Ill-posedness for the Navier-Stokes equations in critical Besov spaces $\dot{B}^{-1}_{\infty,q}$}, Adv. Math. 268 (2015), 350--372.

\bibitem{y-Yoneda10-ill-NS-BMO-1}
T. Yoneda, {\em Ill-posedness of the 3D-Navier-Stokes equations in a generalized Besov space near $BMO^{-1}$}, J. Funct. Anal. 258 (2010), 3376--3387.

			\end{thebibliography}
			\end{document}